\documentclass[12pt]{amsart}
\usepackage{amssymb}
\usepackage{xcolor}
\begin{document}
\numberwithin{equation}{section}

\def\1#1{\overline{#1}}
\def\2#1{\widetilde{#1}}
\def\3#1{\widehat{#1}}
\def\4#1{\mathbb{#1}}
\def\5#1{\frak{#1}}
\def\6#1{{\mathcal{#1}}}

\newcommand{\UH}{\mathbb{H}}
\newcommand{\de}{\partial}
\newcommand{\R}{\mathbb R}
\newcommand{\Ha}{\mathbb H}
\newcommand{\al}{\alpha}
\newcommand{\tr}{\widetilde{\rho}}
\newcommand{\tz}{\widetilde{\zeta}}
\newcommand{\tk}{\widetilde{C}}
\newcommand{\tv}{\widetilde{\varphi}}
\newcommand{\hv}{\hat{\varphi}}
\newcommand{\tu}{\tilde{u}}
\newcommand{\tF}{\tilde{F}}
\newcommand{\debar}{\overline{\de}}
\newcommand{\Z}{\mathbb Z}
\newcommand{\C}{\mathbb C}
\newcommand{\Po}{\mathbb P}
\newcommand{\zbar}{\overline{z}}
\newcommand{\G}{\mathcal{G}}
\newcommand{\So}{\mathcal{S}}
\newcommand{\Ko}{\mathcal{K}}
\newcommand{\U}{\mathcal{U}}
\newcommand{\B}{\mathbb B}
\newcommand{\oB}{\overline{\mathbb B}}
\newcommand{\Cur}{\mathcal D}
\newcommand{\Dis}{\mathcal Dis}
\newcommand{\Levi}{\mathcal L}
\newcommand{\SP}{\mathcal SP}
\newcommand{\Sp}{\mathcal Q}
\newcommand{\A}{\mathcal O^{k+\alpha}(\overline{\mathbb D},\C^n)}
\newcommand{\CA}{\mathcal C^{k+\alpha}(\de{\mathbb D},\C^n)}
\newcommand{\Ma}{\mathcal M}
\newcommand{\Ac}{\mathcal O^{k+\alpha}(\overline{\mathbb D},\C^{n}\times\C^{n-1})}
\newcommand{\Acc}{\mathcal O^{k-1+\alpha}(\overline{\mathbb D},\C)}
\newcommand{\Acr}{\mathcal O^{k+\alpha}(\overline{\mathbb D},\R^{n})}
\newcommand{\Co}{\mathcal C}
\newcommand{\Hol}{{\sf Hol}}
\newcommand{\Aut}{{\sf Aut}(\mathbb D)}
\newcommand{\D}{\mathbb D}
\newcommand{\oD}{\overline{\mathbb D}}
\newcommand{\oX}{\overline{X}}
\newcommand{\loc}{L^1_{\rm{loc}}}
\newcommand{\la}{\langle}
\newcommand{\ra}{\rangle}
\newcommand{\thh}{\tilde{h}}
\newcommand{\N}{\mathbb N}
\newcommand{\kd}{\kappa_D}
\newcommand{\Hr}{\mathbb H}
\newcommand{\ps}{{\sf Psh}}
\newcommand{\Hess}{{\sf Hess}}
\newcommand{\subh}{{\sf subh}}
\newcommand{\harm}{{\sf harm}}
\newcommand{\ph}{{\sf Ph}}
\newcommand{\tl}{\tilde{\lambda}}
\newcommand{\gdot}{\stackrel{\cdot}{g}}
\newcommand{\gddot}{\stackrel{\cdot\cdot}{g}}
\newcommand{\fdot}{\stackrel{\cdot}{f}}
\newcommand{\fddot}{\stackrel{\cdot\cdot}{f}}

\def\Re{{\sf Re}\,}
\def\Im{{\sf Im}\,}

%Pavel's macros
\newcommand{\Real}{\mathbb{R}}
\newcommand{\Natural}{\mathbb{N}}
\newcommand{\Complex}{\mathbb{C}}
\newcommand{\ComplexE}{\overline{\mathbb{C}}}
\newcommand{\Int}{\mathbb{Z}}
\newcommand{\UD}{\mathbb{D}}
\newcommand{\clS}{\mathcal{S}}
\newcommand{\gtz}{\ge0}
\newcommand{\gt}{\ge}
\newcommand{\lt}{\le}
\newcommand{\fami}[1]{(#1_{s,t})}
\newcommand{\famc}[1]{(#1_t)}
\newcommand{\ts}{t\gt s\gtz}
\newcommand{\classCC}{\tilde{\mathcal C}}
\newcommand{\classS}{\mathcal S}

\newcommand{\Unit}[1]{\begin{trivlist}\item\noindent{#1}}
\newcommand{\Step}[1]{\Unit{\bf Step~#1.}}
\newcommand{\step}[1]{\Unit{\underline{Step~#1}.}}
\def\endunit{\end{trivlist}}
\let\endstep=\endunit
\newcommand{\proofbox}{\hfill$\Box$}
\newcommand{\Case}[1]{\Unit{\it Case #1.}}

\newcommand{\mcite}[1]{\csname b@#1\endcsname}
\newcommand{\UC}{\mathbb{T}}

\newcommand{\Moeb}{\mathrm{M\ddot ob}}

\newcommand{\dAlg}{{\mathcal A}(\UD)}
\newcommand{\diam}{\mathrm{diam}}

\theoremstyle{theorem}
\newtheorem {result} {Theorem}
\setcounter {result} {64}
 \renewcommand{\theresult}{\char\arabic{result}}

%\newcommand{\Step}[2]{\begin{itemize}\item[{\bf Step~#1.}]{\it #2}\end{itemize}}

%\newcommand{\proofbox}{\hfill$\Box$}
%end of Pavel's macros block

%New macros Pavel
\newcommand{\Spec}{\Lambda^d}
\newcommand{\SpecR}{\Lambda^d_R}
\newcommand{\Prend}{\mathrm P}
\newcommand{\classM}{\mathbb M}
\newcommand{\Log}{\mathop{\mathrm{Log}}}

\def\ch{\mathop{\mathrm{ch}}}
\def\th{\mathop{\mathrm{th}}}
\def\sh{\mathop{\mathrm{sh}}}
\def\tg{\mathop{\mathrm{tg}}}
\newcommand{\CarClass}{{\mathcal C}}
\newcommand{\ParClass}{{\mathcal V}}
%end of New macros Pavel

%\tableofcontents

% Standard sets

\def\cn{{\C^n}}
\def\cnn{{\C^{n'}}}
\def\ocn{\2{\C^n}}
\def\ocnn{\2{\C^{n'}}}
\def\je{{\6J}}
\def\jep{{\6J}_{p,p'}}
\def\th{\tilde{h}}

% Abbreviations

\def\dist{{\rm dist}}
\def\const{{\rm const}}
\def\rk{{\rm rank\,}}
\def\id{{\sf id}}
\def\aut{{\sf aut}}
\def\Aut{{\sf Aut}}
\def\CR{{\rm CR}}
\def\GL{{\sf GL}}
\def\Re{{\sf Re}\,}
\def\Im{{\sf Im}\,}
\def\U{{\sf U}}

\def\la{\langle}
\def\ra{\rangle}

\emergencystretch15pt \frenchspacing

\newtheorem{theorem}{Theorem}[section]
\newtheorem{lemma}[theorem]{Lemma}
\newtheorem{proposition}[theorem]{Proposition}
\newtheorem{corollary}[theorem]{Corollary}

\theoremstyle{definition}
\newtheorem{definition}[theorem]{Definition}
\newtheorem{example}[theorem]{Example}

\theoremstyle{remark}
\newtheorem{remark}[theorem]{Remark}
\numberwithin{equation}{section}

\long\def\REM#1{\relax}

\newcommand{\func}[1]{\mathop{\mathrm{#1}}}
\renewcommand{\Re}{\func{Re}}
\renewcommand{\Im}{\func{Im}}

\newenvironment{mylist}{\begin{list}{}%
{\labelwidth=2em\leftmargin=\labelwidth\itemsep=.4ex plus.1ex minus.1ex\topsep=.7ex plus.3ex
minus.2ex}%
\let\itm=\item\def\item[##1]{\itm[{\rm ##1}]}}{\end{list}}

\title[Loewner Theory in annulus I]{Loewner Theory in annulus I: evolution families and
differential equations}

\author[M. D. Contreras]{Manuel D. Contreras $^\dag$}

\author[S. D\'{\i}az-Madrigal]{Santiago D\'{\i}az-Madrigal $^\dag$}
\address{Camino de los Descubrimientos, s/n\\
Departamento de Matem\'{a}tica Aplicada II\\
Escuela T\'{e}cnica Superior de Ingenier\'\i a\\
Universidad de Sevilla\\
Sevilla, 41092\\
Spain.}\email{contreras@us.es} \email{madrigal@us.es}

\author[P. Gumenyuk]{Pavel Gumenyuk $^\ddag$}
\address{Department of
Mathematics\\ University of Bergen, Johannes Brunsgate 12\\ Bergen 5008, Norway. }
\email{Pavel.Gumenyuk@math.uib.no}

%\expandafter\def\csname subjclassname@2010\endcsname{\textup{2010} Mathematics Subject
%Classification}

%\thanks{Agradecimientos}
\date{\today }


\thanks{$^\dag\,^\ddag$ Partially supported by the ESF Networking Programme ``Harmonic and Complex Analysis
and its Applications'' and by \textit{La Consejer\'{\i}a de Econom\'{\i}a, Innovaci\'{o}n y Ciencia
de la Junta de Andaluc\'{\i}a} (research group FQM-133)}

\thanks{$^\dag$ Partially supported by the \textit{Ministerio
de Ciencia e Innovaci\'on} and the European Union (FEDER), project MTM2009-14694-C02-02}

\thanks{$^\ddag$ Supported by a grant from Iceland, Liechtenstein and Norway through the EEA Financial Mechanism.
Supported and coordinated by Universidad Complutense de Madrid and by Instituto de Matem\'{a}ticas
de la Universidad de Sevilla. Partially supported by the {\it Scandinavian Network ``Analysis
and Applications''} (NordForsk), project \#080151,%
\ and the {\it Research Council of Norway}, project \#177355/V30}

\begin{abstract}
Loewner Theory, based on dynamical viewpoint, is a powerful tool in Complex Analysis, which plays a
crucial role  in such important achievements as the proof of famous Bieberbach's conjecture and
well-celebrated Schramm's Stochastic Loewner Evolution~(SLE). Recently Bracci et al
\cite{BCM1,BCM2,SMP} have proposed a new approach bringing together all the variants of the
(deterministic) Loewner Evolution in a simply connected reference domain. We construct an analogue
of this theory for the annulus. In this paper, the first of two articles, we introduce a general
notion of an evolution family over a system of annuli and prove that there is a 1-to-1
correspondence between such families and semicomplete weak holomorphic vector fields. Moreover, in
the non-degenerate case, we establish a constructive characterization of these vector fields
analogous to the non-autonomous Berkson\,--\,Porta representation of Herglotz vector fields in the
unit disk~\cite{BCM1}.

\end{abstract}

\maketitle

\tableofcontents

\section{Introduction}

\subsection{Loewner Theory}
In recent years, there has been a drastic growth of interest in the dynamical aspects in Complex
Analysis. Most of all, this applies to the so-called {\it Loewner Theory}, dealing with the
semigroup $\Hol(\UD,\UD)$ of all holomorphic self-maps of the unit disk~$\UD:=\{z:|z|<1\}$.

The core of modern Loewner Theory resides in the connection and interplay of the following three
basic notions:

\begin{mylist}
\item[-] {\it Herglotz vector fields $G:\UD\times[0,+\infty)\to\Complex$}, which are semicomplete time-dependent
holomorphic vector fields in the unit disk~$\UD$ and can be described as integrable families
$(G_t)_{t\ge0}$ of infinitesimal generators of one-parametric semigroups in~$\Hol(\UD,\UD)$;

\item[-] {\it Evolution families $(\varphi_{s,t})_{t\ge s\ge0}$}, which can be characterized as non-autonomous
holomorphic semiflows generated by Herglotz vector fields;

\item[-] {\it Loewner chains $(f_t)_{t\ge0}$}, which are one-parametric families of univalent functions
${f_t:\UD\to\Complex}$ with expanding systems of image domains~$f_t(\UD)$. Any Loewner chain
satisfies a linear PDE, known as {\it (generalized) Loewner\,--\,Kufarev PDE}, driven by a Herglotz
vector field. The corresponding evolution family $\varphi_{s,t}:=f_t^{-1}\circ f_s$ can be obtained
by solving the characteristic equation for this PDE.
\end{mylist}
There is an {\it essentially one-to-one correspondence} between Herglotz vector fields, evolution
families, and Loewner chains.

Classical Loewner Theory originated from the so-called {\it Parametric Representation} of univalent
functions proposed in~1923 by C.\,Loewner~\cite{Loewner} and developed further by a number of
specialists in Geometric Function Theory, among which we would like to mention the fundamental
contributions of Kufarev~\cite{Kufarev} and
Pommerenke~\cite{Pommerenke-65},\,\cite[Ch.\,6]{Pommerenke}. In the case which they considered, and
which is usually referred in modern literature to as the {\it(classical) radial case}, a Loewner
chain is a family of univalent functions $f_t(z)=e^tz+a_2(t) z^2+\ldots$, $z\in\UD$, such that
$f_s(\UD)\subset f_t(\UD)$ whenever $t\ge s$. A (classical radial) Herglotz vector field is of the
form $G(w,t):=-wp(w,t)$, where $p$ is holomorphic in $z\in\UD$, measurable in $t\ge0$, and
satisfies conditions $p(0,t)=1$ and $\Re p(w,t)>0$ for all $z\in \UD$ and a.e. $t\ge0$. For any
holomorphic vector field of this form there exists a unique classical radial Loewner chain $(f_t)$
such that $[s,+\infty)\ni t\mapsto\varphi_{s,t}(z):=f_t^{-1}\circ f_s(z)$ solves, for any $z\in\UD$
and $s\ge0$, the Loewner\,--\,Kufarev ODE $\dot w=G(w,t)$ with the initial condition~$w|_{t=s}=z$.
The converse is also true: any classical radial Loewner chain is a solution to the
Loewner\,--\,Kufarev PDE $\partial f_t(z)/\partial t=-f'(z)G(z,t)$ driven by some classical radial
Herglotz vector field~$G$, while the evolution family~$(\varphi_{s,t})$ corresponding to~$(f_t)$
solves the Loewner\,--\,Kufarev ODE. Moreover, the Loewner chain~$(f_t)$ can be reconstructed via
its evolution family~$(\varphi_{s,t})$ by means of the formula~$f_s=\lim_{t\to+\infty}
e^{t}\varphi_{s,t}$.

This relation between $G$, $(f_t)$, and $(\varphi_{s,t})$ provides a representation of the
class~$\mathcal S$ of all normalized holomorphic univalent functions in~$\UD$, since any
$f_0\in\mathcal S$ can be embedded as an initial element into a classical radial Loewner chain
\cite[Th.\,6.5 on p.\,159]{Pommerenke},~\cite{Gutljanski}. This representation was one of the main
tools in proof of the famous Bieberbach conjecture, given by de Branges~\cite{deBranges}.

A similar representation, designated in modern terminology by the attribute {\it chordal}, was
proposed by Kufarev and his students  for holomorphic univalent self-maps of the upper half-plane
with the hydrodynamic normalization at the point of infinity \cite{Kufarev_etal}; see also
references cited in~\cite[p.\,543--544]{SMP2}.

Komatu~\cite{Komatu} was the first who was able to apply Loewner's ideas for parametric
representation of univalent functions in the annulus. His approach was developed by
Goluzin~\cite{GoluzinM}, Li En Pir~\cite{LiEnPir}, Lebedev~\cite{Lebedev}, and
Gutljanski\u\i~\cite{Gutljanski_diss}.  More general cases of Loewner Evolution in multiply
connected context were studied in~\cite{Komatu1950} and by an essentially different method
in~\cite{KufarevKuvaev, Tsai}. The monograph~\cite{Aleksandrov} contains a self-contained detailed
account on the Parametric Representation both in simply and multiply connected cases.

Until the last decade the attention of specialists in Loewner Theory was mainly paid to the radial
case in the unit disk, first of all because of its applications in the study of the class~$\mathcal
S$ and its subclasses. The significance of the chordal Loewner Evolution as well as the Loewner
Evolution in multiply connected domains was apparently underestimated. However, nowadays the
Parametric Method, invented by Loewner to study the Bieberbach conjecture, has gone far beyond the
scope of the original problem, and the distribution of active interest in various aspects of the
theory has been changed. The most spectacular evidence of this fact is the well-celebrated
Stochastic Loewner Evolution (SLE) invented in 2000 by Schramm~\cite{Schramm}. It appears that the
chordal Loewner\,--\,Kufarev equation driven by the Brownian motion is intrinsically related to
several important lattice models in Statistical Physics, such as critical Ising model. We refer
interested readers to~\cite{LawlerBulAMS}. For a wider discussion on connection of deterministic
and stochastic Loewner Evolutions to Conformal Field Theory and Integrable Systems
see~\cite{VasMark-Sevilla} and references cited there.

To conclude, we would like to mention the survey paper~\cite{letters} that covers the basics of the
classical and modern Loewner Theory, its history, recent development and applications.

\subsection{Problem definition and main results}
Recently Bracci and the first two authors of this paper~\cite{BCM1,BCM2} introduced a new general
approach in Loewner Theory in the unit disk, which unifies, and contains as very special cases,
both chordal and radial Loewner Evolutions. This approach is based on a general intrinsic notion of
evolution family in~$\UD$ (see Definition~\ref{D_EV-simply}), which can be viewed as non-autonomous
generalization of continuous one-parametric semigroups in~$\Hol(\UD,\UD)$. Further developments in
this direction can be found in~\cite{EFversusSG,conTiz,SMP,SMP2}. In this paper we address the
following
\Unit{\bf Problem:} {\it to construct a general Loewner Theory in the annulus.}
\endunit
Our motivation is based on two reasons. Firstly, although this abstract approach was generalized to
arbitrary finite-dimensional complete hyperbolic complex manifolds~\cite{LA-FB,BCM2}, applying it
directly to any {\it multiply connected} hyperbolic Riemann surface~$D$ which is not conformally
equivalent to the punctured disk~$\UD^*:=\UD\setminus\{0\}$ would lead to a quite trivial theory,
because the connected component of~$\Hol(D,D)$ containing~$\id_D$ coincides with the group of
rotations  if $D$ is doubly connected, or consists of $\id_D$ only otherwise (see, e.g., \cite[\S
1.2.2]{Abate}). It follows that in order to develop an interesting substantial theory for multiply
connected case, instead of a static reference domain or Riemann surface~$D$ one has to consider a
one-parametric family~$(D_t)_{t\ge0}$ of reference domains admitting existence of injective
holomorphic mappings $\varphi_{s,t}:D_s\to D_t$ homotopically equivalent to the identity for any
$s\ge0$ and any $t\ge s$. In doubly connected case it leads to a family of expanding annuli
$D_t=\mathbb A_{r(t)}:=\{z:r(t)<|z|<1\}$, where ${r:[0,+\infty)\to[0,1)}$ is non-increasing and
continuous. The first who implemented this idea was already Komatu~\cite{Komatu}, but up to our
best knowledge all the previous studies of Loewner Evolution in doubly and, more generally,
multiply connected domains deal only with some special cases.

The other reason is the recent boost of interest in Loewner Theory as a whole as well as to its
variants for multiply connected domains. As an illustration, we cite papers
[\mcite{mSLE1}\,--\,\mcite{mSLE3},\,\mcite{Zhan}] extending the notion of SLE to multiply connected
case.

The study we present here is intended to be the first of two papers on the problem stated above; it
contains the theory of evolution families and semicomplete weak holomorphic vector
fields\footnote{The class of semicomplete weak holomorphic vector fields in~$\UD$ conincides with
that of Herglotz vector fields. However, we prefer to avoid using the term  Herglotz vector field
in doubly connected context.}, while Loewner chains will be considered in another paper.

In Section~\ref{S_evolDoubly} we introduce a notion of {\it evolution family over a canonical
domain system} of annuli (see Definitions~\ref{def-cansys} and~\ref{def-ev}), analogous in a
certain sense to that proposed in~\cite{BCM1} for the unit disk,  and establish some basic
properties of these evolution families.

In Section~\ref{S_EF-sWHVF} we discuss relationship between evolution families and Carath\'eodory
ODEs. In particular, we prove that {\it there is a 1-to-1 correspondence between the evolution
families we have introduced and semicomplete weak holomorphic vector fields}. More precisely, {\it
every evolution family~$\big((D_t),(\varphi_{s,t})\big)$ of order $d\in[1,+\infty]$ can be
described as the non-autonomous semiflow of some semicomplete weak holomorphic vector field $G$ of
order~$d$ in $\mathcal D:=\{(z,t):\,t\ge0,\,z\in D_t\}$. Conversely, every such semiflow
constitutes an evolution family}. See Theorem~\ref{TH_EF<->sWHVF} and Definitions~\ref{D_WHVF}
and~\ref{D_semicomplete} for the exact formulation.

Further, in Theorem~\ref{TH_semi-char-non-deg} we establish, for the case when all the annuli $D_t$
are non-degenerate, a {\it constructive characterization of semicomplete weak holomorphic vector
fields in $\mathcal D$} analogous to the non-autonomous version of the Berkson\,--\,Porta
representation~\cite[Theorem~4.8]{BCM1}, which characterizes Herglotz vector fields in the unit
disk.

A part of our proofs rely on lifting evolution families from a system of annuli to a simply
connected domain~$D$. This technique, together with some new results on evolution families in~$D$,
is developed in Section~\ref{S_evolSimply}.

Moreover, we include auxiliary Section~\ref{S_ODE} on Carath\'eodory differential equations driven
by weak holomorphic vector fields. The first part contains standard facts about solutions to such
ODEs which we regard as known. However, we include a sketch of the proofs, because up to our best
knowledge, no literature gives a direct proof of these results formulated exactly as we need. The
second part is devoted to the case of the ODEs which are semicomplete to the right. It contains an
exact characterization of the solutions, which is applied later to prove
Theorem~\ref{TH_EF<->sWHVF}.

In Section~\ref{S_examples} we consider a number of examples. In particular, we prove that, in
contrast to the simply connected case, an evolution family over a system of annuli can share any
finite number of fixed points.

Finally, in Section~\ref{S_Lebedev} we discuss shortly how our results are related to the
achievements of Komatu, Goluzin, Li En Pir, and Lebedev.

\section{Holomorphic Carath\'eodory differential equations}
\label{S_ODE}

In this section we obtain a characterization of solutions to  Carath\'eodory ordinary differential
equations driven by holomorphic vector fields.

Let $E\subset\Real$ be a non-empty interval, bounded or unbounded, and consider a connected
relatively open subset $\mathcal D$ of the set~$\Complex\times E$. We fix $E$ and $\mathcal D$
throughout this section. In this paper we apply the results of this section for the case
$E:=[0,+\infty)$ and $\mathcal D:=\UD\times[0,+\infty)$ or $\mathcal
D:=\{(z,t):\,t\ge0,\,0<|z|<r(t)\}$, where $r:[0,+\infty)\to[0,1)$ is non-increasing and continuous.
However, we prefer to consider the general case, which might be useful for further studies and at
the same time does not require any substantial modification of the argument.

By $AC^d(X,Y)$, where $X\subset \R$ and $d\in[1,+\infty]$, we denote the class of all {\it locally}
absolutely continuous functions $f:X\to Y$ such that the derivative $f'$ belongs to $L_{\rm
loc}^d(X)$. Further, denote by $D(z_0,r)$, where $r>0$ and $z_0\in\Complex$, the disk
$\{z:|z-z_0|<r\}$, and let $\UD:=D(0,1)$ and $\UC:=\partial\UD$.

\subsection{Weak holomorphic vector fields and Carath\'eodory ODEs}

\def\pr{\mathrm{pr}}
Carath\'eodory's theory of ODEs is a well-established area of Analysis, see
e.g. \cite[\S II.1]{Coddington}, \cite[\S I.1]{Filippov},
\cite[Ch.\,18]{Kurzweil}, or \cite[\S VIII.8]{Sansone}. The facts we state
below can be regarded as well-known. However, these results do not seem to be
proved earlier in the literature directly and explicitly in the form that our
context requires. That is why we prefer to sketch the proofs.

We start by introducing the class of vector fields we deal with.

\begin{definition}\label{D_WHVF}
Let  $d\in[1,+\infty]$.  A function $G:\mathcal D\to\Complex$ is said to be a {\it weak holomorphic
vector field} of order~$d$ in the domain $\mathcal D$, if it satisfies the following conditions:
\begin{mylist}
\item[WHVF1.] For each $z\in\Complex$ the function $G(z,\cdot)$ is measurable in~$E_z:=\{t:(z,t)\in\mathcal
D\}$.
\item[WHVF2.] For each $t\in E$ the function $G(\cdot,t)$ is holomorphic in
$D_t:=\{z:(z,t)\in\mathcal D\}$.

\item[WHVF3.] For each compact set $K\subset\mathcal D$ there exists a non-negative function ${k_K\in
L^d\big(\pr_{\mathbb R}(K),\Real\big)}$, where $\pr_{\mathbb R}(K):=\{t\in
E:~\exists\,z\in\Complex\quad(z,t)\in K\}$, such that
$$%
|G(z,t)|\le k_K(t),\quad\text{for all $(z,t)\in K$}.
$$%
\end{mylist}
\end{definition}
\begin{remark}\label{RM_changeWHVF3}
It is easy to see that condition WHVF3 in Definition~\ref{D_WHVF} is equivalent to
\begin{mylist}
\item[WHVF3'] For any closed disk $B\subset \Complex$ and any compact interval $I\subset E$ such
that $B\times I\subset \mathcal D$ there exists a non-negative function $k_{B,I}\in L^d(I,\Real)$
such that $|G(z,t)|\le k_{B,I}(t)$ for all $t\in I$ and all $z\in B$.
\end{mylist}
\end{remark}

Given a weak holomorphic vector field~$G$ in~$\mathcal D$ and a point $(z,s)\in\mathcal D$, let us
consider the following initial value problem for a Carath\'eodory ODE
\begin{equation}\label{EQ_CarODE-IVP}
\dot w=G(w,t),\quad w(s)=z.
\end{equation}
Problem~\eqref{EQ_CarODE-IVP} is equivalent (see, e.\,g.,~\cite[\S I.1]{Filippov}) to the integral
equation $\mathcal L_z^s w=w$, where
\begin{equation}\label{EQ_opL}
(\mathcal L_z^s w)(t):=z+\int_{s}^t G(w(\xi),\xi)\,d\xi
\end{equation}
is well-defined for any continuous function $w:J\to\Complex$ such that $J\subset E$ is an interval,
$s\in J$, and $(w(t),t)\in\mathcal D$ for all $t\in J$.

\begin{theorem}\label{TH_CarODETheory}
Let $G$ be a weak holomorphic vector field in $\mathcal D$ of order~$d\in[1,+\infty]$. Then the
following statements hold:
\begin{mylist}
\item[(i)] For any $(z,s)\in\mathcal D$, the initial value problem~\eqref{EQ_CarODE-IVP} has a unique non-extendable solution
$t\mapsto w^*_s(z,t)$. In other words, there exists an interval $J_*(z,s)\subset E$, $s\in
J_*(z,s)$, and a solution $w^*_s(z,\cdot)$ to the problem~\eqref{EQ_CarODE-IVP} defined in
$J_*(z,s)$ such that any other solution $t\mapsto w_0(t)$ to~\eqref{EQ_CarODE-IVP} is a restriction
of $w^*_s(z,\cdot)$.

\item[(ii)] For any $(z,s)\in\mathcal D$ the non-extendable solution $w^*_s(z,\cdot)$ belongs
to~$AC^d\big(J_*(z,s),\Complex\big)$.

\item[(iii)] For any $(z,s)\in\mathcal D$ there exists no compact set $K\subset\mathcal D$ such that
$\big\{\big(w^*_s(z,t),\,t\big):~t\in J_*(z,s),\,t\ge s\big\}\subset K$ unless $\sup J_*(z,s)=\sup
E$.

\item[(iv)] For any $(z_0,s_0)\in\mathcal D$ and any $t_0\in J_*(z_0,s_0)$ there exists $\varepsilon>0$ such
that $t_0\in J_*(z,s_0)$ whenever $|z-z_0|<\varepsilon $. Moreover, the function $z\mapsto
w^*_{s_0}(z,t_0)$ is holomorphic and injective in~$D(z_0,\varepsilon)$.

\item[(v)] For any $(z_0,s_0)\in\mathcal D$ and any $t_0\in J_*(z_0,s_0)$ there exists $\epsilon>0$
such that the mapping $(z,s,t)\mapsto w^*_{s}(z,t)$ is well-defined for all $(z,s,t)\in
U(\epsilon):=\{(z,s,t):\,{z\in D(z_0,\epsilon),}~{s,t\in E,}\, {|t-t_0|<\epsilon,}\,
{|s-s_0|<\epsilon}\}$. Moreover, this mapping is continuous at the point~$(z_0,s_0,t_0)$.
\end{mylist}
\end{theorem}
First we prove the following lemma.

\begin{lemma}\label{LM_CarODE-local}
Let $G:\mathcal D\to\Complex$ be a weak holomorphic vector field of order~$d\in[1,+\infty]$ and
$K\subset\mathcal D$ a compact set. Suppose $I\subset E$ is a compact interval and
$B:=\overline{D(z_0,2R)}$ is a closed disk such that $B\times I\subset K$. Then for any $s\in I$
the following assertions hold:
\begin{mylist}
\item[(a)] for each $z\in D(z_0,R)$ and each sufficiently small interval $J\subset I$, $J\ni s$,
there exists a unique solution $J\ni t\mapsto w_s(z,t)$ to the initial value
problem~\eqref{EQ_CarODE-IVP};
\item[(b)] for each $z\in D(z_0,R)$ the solution to~\eqref{EQ_CarODE-IVP} can be continued
all over the interval $J_h(s):=(s-h,s+h)\cap I$, where $h=h(G,K,R)>0$ is a constant independent
of~$I$, $s$, and $z$;
\item[(c)] for each $t\in J_h$ the function $w_s(\cdot,t)$ is holomorphic in $D(z_0,R)$.
\end{mylist}
Moreover,
\begin{mylist}
\item[(d)] the mapping $(z,s,t)\mapsto w_s(z,t)$ is continuous on $A:=\{(z,s,t):\,{z\in D(z_0,R),}\, {s,t\in
I,}\,\\ {|s-t|<h}\}$.
\end{mylist}
\end{lemma}

\begin{proof}
First of all let us note that, since $K\subset\mathcal D$ is compact, there exists $\delta>0$ such
that
$$%
\bigcup_{(z,t)\in K}\Big(\overline{D(z,\delta)}\times\{t\}\Big)\subset\subset\mathcal D.
$$%
Then the argument of the proof of \cite[Lemma~4.2]{BCM1} can be easily adapted
to show that there exists a non-negative function $\hat k_K\in
L^d\big(\pr_{\mathbb R}(K),\Real\big)$ depending only on $G$ and $K$ such that
\begin{equation}\label{EQ_G-diff}
|G(z_2,t)-G(z_1,t)|\le \hat k_K(t)|z_2-z_1|\quad\text{whenever $(z_1,t), (z_2,t)\in K$.}
\end{equation}

Choose any $\alpha\in(0,1)$. There exists $h>0$ that fulfills the following two conditions:
\begin{equation}\label{EQ_cond1}
\left|\int_{s}^{t}k_K(\xi)\,d\xi\right|<R\quad\text{whenever $s,t\in I$ and $|t-s|<h$},
\end{equation}
\begin{equation}\label{EQ_cond2}
\left|\int_{s}^{t}\hat k_K(\xi)\,d\xi\right|\le\alpha\quad\text{whenever $s,t\in I$ and $|t-s|<h$}.
\end{equation}

By $C(X,Y)$ let us denote the class of all continuous functions from $X$ to~$Y$. Fix any $s\in I$.
Given an interval $J\subset J_h(s)$ with $s\in J$, from \eqref{EQ_G-diff}\,--\,\eqref{EQ_cond2} it
follows easily that for any $z\in D(z_0,R)$, the operator $\mathcal L_z^s$, given
by~\eqref{EQ_opL}, is a contracting self-mapping of~$C(J,B)$ endowed with the Chebyshov metric
$\rho(w_2,w_1):=\sup_{t\in J}|w_2(t)-w_1(t)|$. The metric space $M:=\big(C(J,B),\rho\big)$ is
complete. Hence  the Banach fixed point theorem implies statements (a) and (b) of the lemma.

The proof of (d) is similar. The operator $\mathcal L$, defined by the formula $\big(\mathcal L
w\big)(z,t;s):=\big(\mathcal L_z^s w(z,\cdot;s)\big)(t)$, is a contractive self-mapping of $C(A,B)$
endowed with the metric $\tilde \rho(w_2,w_1):=\sup_{(z,s,t)\in A}\big|w_2(z,t;s)-w_1(z,t;s)\big|$.
The metric space $\tilde M=\big(C(A,B),\tilde \rho\big)$ is complete. Hence it contains a solution
to the equation $\mathcal Lw=w$, which, in virtue of~(a), has to coincide with $A\ni (z,s,t)\mapsto
w_s(z,t)$. This proves (d).

We are left with statement~(c). Clearly, it is sufficient to show that $\mathcal L$ maps into
itself the closed subspace $\tilde M_{\rm hol}$ of $\tilde M$ consisting of all maps from~$\tilde
M$ which are holomorphic in~$z$.

So consider any $w\in\tilde M_{\rm hol}$. Fix an arbitrary $(z,s,t)\in A$. For all $\xi\in J_h(s)$
and all $\omega\in U:=D\big(0,(R-|z-z_0|)/2\big)\setminus\{0\}$ we have
\begin{equation}\label{EQ_LebDomConv}%
\left|\frac{G\big(w(z+\omega,\xi;s),\xi\big)-G\big(w(z,\xi;s),\xi\big)}{\omega}\right|\le C\,\hat
k_K(\xi)
\end{equation}
with some constant $C=C(z)>0$ not depending on $\omega$ and $\xi$. Here we
combined~\eqref{EQ_G-diff} with the fact that the family $\big\{w(\cdot,\xi;s):\xi\in J_h(s)\big\}$
is normal in~$D(z_0,R)$ and hence the mapping $(\omega,\xi)\mapsto
\big(w(z+\omega,\xi;s)-w(z,\xi;s)\big)/\omega$ is bounded on $U\times J_h(s)$.

By construction,
\begin{multline}\label{EQ_opLdiff}
\frac{\big(\mathcal L w\big)(z+\omega,t;s)-\big(\mathcal L
w\big)(z,t;s)}{\omega}=1+\\+\int_{s}^{t}\frac{G\big(w(z+\omega,\xi;s),\xi\big)-
G\big(w(z,\xi;s),\xi\big)}{\omega}\,d\xi.
\end{multline}

From~\eqref{EQ_LebDomConv} it follows that we can apply Lebesgue's dominated convergence theorem to
conclude that there exists a finite limit of the integral in~\eqref{EQ_opLdiff} as $\omega\to0$.
Thus $\big(\mathcal L w)(\cdot,t;s)$ is differentiable in the complex sense in $D(z_0,R)$. This
completes the proof.
\end{proof}

\begin{proof}[\bf Proof of Theorem~\ref{TH_CarODETheory}]
Assertion (a) of Lemma~\ref{LM_CarODE-local} implies that for any
$(z,s)\in\mathcal D$ there exists a unique local solution
to~\eqref{EQ_CarODE-IVP}. Thus statement~(i) of the theorem follows in the same
way as in the classical theory of ODEs.

Fix now any $(z,s)\in\mathcal D$ and take an arbitrary compact interval $I\subset J_*(z,s)$. By
condition WHVF3 with $K:=\big\{\big(w_s^*(z,\xi),\,\xi\big):~\xi\in I\big\}$, for any $t_1,t_2\in
I$ such that $t_1<t_2$ we have
$$|w_s^*(z,t_2)-w_s^*(z,t_1)|\le
\int_{t_1}^{t_2}\big|G\big(w_s^*(z,\xi),\xi\big)|\,d\xi\le\int_{t_1}^{t_2}k_K(\xi)\,d\xi,$$ where
$k_K\in L^d(I,\Real)$ and does not depend on $t_1$ and $t_2$. This proves statement~(ii).

Let $K\subset\mathcal D$ be any compact set. Then there is another compact set~$K_1\subset\mathcal
D$ and constants $R>0$ and $\delta>0$ such that $\overline
{D(z,2R)}\times\big([s-\delta,s+\delta]\cap E\big)\subset K_1$ whenever $(z,s)\in K$. We claim that
any solution $J\ni t\mapsto w_0(t)$ to the equation $\dot w=G(w,t)$ such that
$\big(w_0(t),t\big)\in K$ for all $t\in J$, is extendable to a neighborhood of $t^*:=\sup J$
provided $t^*<\sup E$. Indeed, by assertion~(b) of Lemma~\ref{LM_CarODE-local} applied with $K_1$
substituted for~$K$, there exists $\delta_1\in(0,\delta]$ such that for any $s\in J$ the initial
value problem~\eqref{EQ_CarODE-IVP} with $z:=w_0(s)$ has a solution defined on the interval
$(s-\delta_1,s+\delta_1)\cap E$. In view of the uniqueness of solution to~\eqref{EQ_CarODE-IVP},
this proves our claim, which, in its turn, implies statement~(iii) of the theorem.

Let us prove statement~(iv). For simplicity we may assume that $t_0>s_0$. Let $J:=[s_0,t_0]$.
Taking $w_0:=w_{s_0}^*(z_0,\cdot)|_J$ and $K:=\big\{\big(w_0(t),t\big):\,\,t\in J\big\}$ in the
above argument and using assertion~(c) of Lemma~\ref{LM_CarODE-local} we can conclude that there
exists a finite increasing sequence $(s_j)_{j=0}^n$ starting with $s_0$ and finishing with
$s_n:=t_0$ such that for each $j=1,\ldots,n$ the function $w^*_{s_{j-1}}(z,t)$ is well-defined and
holomorphic in $z$ for all $t\in I_j:=[s_{j-1},s_j]$ and all $z\in D(w_0(s_{j-1}),R)$. Since
$w^*_{s_{j-1}}\big(w_0(s_{j-1}),s_j\big)=w_0(s_j)$, there exists $\varepsilon>0$ such that the
composition $f(z):=w^*_{s_{n-1}}(\cdot,s_n)\circ\ldots\circ w^*_{s_1}(\cdot,s_2)\circ
w^*_{s_0}(\cdot,s_1)$ is well-defined and holomorphic in $D(z_0,\varepsilon)$. It follows that
$w_{s_0}^*(\cdot,t_0)$ is also well-defined in $D(z_0,\varepsilon)$ and coincides there with~$f$.
The fact that $w_{s_0}^*(\cdot,t_0)$ is injective follows from the uniqueness of the solution in
the same way as in the classical theory of ODEs.

We are left with the proof of statement~(v). By statement~(iv) the map $w_{s_0}^*(\cdot,t_0)$ is
defined in $D(z_0,\varepsilon)$ and is continuous at the point~$z_0$. It follows from
assertions~(b) and~(d) of Lemma~\ref{LM_CarODE-local} that there exists $\varepsilon_1>0$ such that
the map $(z,s)\mapsto w_{s}(z,s_0)$ is well defined in $\mathcal
O_{\varepsilon_1}(z_0,s_0):=\{(z,s):|z-z_0|<\varepsilon_1,\,|s-s_0|<\varepsilon_1,\,s\in E\}$ and
continuous at~$(z_0,s_0)$. Hence the map $(z,s)\mapsto
f(z,s):=w^*_{s_0}\big(w^*_{s}(z,s_0),t_0\big)$ is well-defined in $\mathcal
O_{\varepsilon_2}(z_0,s_0)$ for some $\varepsilon_2>0$ and continuous at the point $(z_0,s_0)$.
Denote $\zeta_0:=w_{s_0}^*(z_0,t_0)$. Again by assertions~(b) and~(d) of
Lemma~\ref{LM_CarODE-local}, the map $(\zeta,t)\mapsto w^*_{t_0}(\zeta,t)$ is well-defined in
$\mathcal O_{\varepsilon_3}(\zeta_0,t_0)$ for some $\varepsilon_3>0$ and continuous at
$(\zeta_0,t_0)$. Thus the map $(z,s,t)\mapsto g(z,s,t):=w^*_{t_0}(f(z,s),t)$ is well-defined in
$U(\epsilon)$ for some $\epsilon>0$ and continuous at the point~$(z_0,s_0,t_0)$.  To finish now the
proof it remains to notice that $g$ is a restriction of the mapping $(z,s,t)\mapsto w_s^*(z,t)$.
\end{proof}

\subsection{Semicomplete weak holomorphic vector fields and families of holomorphic functions generated by them}
In this section we consider weak holomorphic vector fields~$G$ for which the solution to the
initial value problem~\eqref{EQ_CarODE-IVP} exists globally to the right. Our proofs take advantage
of the methods used in~\cite{BCM1,BCM2,EFversusSG,conTiz}. Without loss of generality we adopt the
following

\Unit{\bf Assumption.} For any $t\in E$ the set $D_t:=\{z:(z,t)\in \mathcal D\}$ is not empty. For
simplicity we will assume that all $D_t$'s are domains, which is enough for our purpose. However,
our arguments (with minimal changes) are also valid for the case when some of the $D_t$'s are not
necessarily connected.
\endstep

Recall that by Theorem~\ref{TH_CarODETheory}, for any $(z,s)\in\mathcal D$ all solutions to the
initial value problem~\eqref{EQ_CarODE-IVP} can be obtained as restrictions of the unique
non-extendable solution $J_*(z,s)\ni t\mapsto w^*_s(z,t)\in\Complex$.
\begin{definition}\label{D_semicomplete}
A weak holomorphic vector field $G:\mathcal D\to\Complex$ (of some order~$d\in[1,+\infty]$) is said
to be {\it semicomplete} if for any $(z,s)\in\mathcal D$, the inclusion $J_*(z,s)\supset
E^s:=E\cap[s,+\infty)$ takes place, i.e. the initial value problem~\eqref{EQ_CarODE-IVP} has a
solution defined everywhere in $E^s$.
\end{definition}

In the following proposition we establish some important properties of the non-autonomous semiflows
generated by semicomplete weak holomorphic vector fields.
\begin{proposition}\label{PR_sol_semicomplete}
Let $G:\mathcal D\to\Complex$ be a semicomplete weak holomorphic vector field of
order~$d\in[1,+\infty]$. Then the formula $\varphi_{s,t}(z):=w_s^*(z,t)$ defines a family
$(\varphi_{s,t})_{s\in E,\,t\in E^s}$ of mappings $\varphi_{s,t}:D_s\to D_t$ such that the
following assertions hold:
\begin{mylist}
\item[(i)] $\varphi_{s,t}$ is holomorphic in $D_s$ for any $s\in E$ and any $t\in E^s$;
\item[(ii)] $\varphi_{s,s}=\id_{D_s}$ for any $s\in E$;
\item[(iii)] $\varphi_{s,t}=\varphi_{u,t}\circ \varphi_{s,u}$ for any $s,u,t\in E$ such that $s\le u\le
t$;
\item[(iv)] for any compact set $K\subset \mathcal D$ there exists a non-negative function $\tilde k_K\in L^d_{\rm
loc}(E,\Real)$ such that
$$%
|\varphi_{s,t}(z)-\varphi_{s,u}(z)|\le\int_{u}^t \tilde k_K(\xi)d\xi
$$%
for any $u,t\in E$ and any $(z,s)\in K$ satisfying $s\le u\le t$.
\end{mylist}
\end{proposition}
\begin{proof}
The proof of the first 3 assertions is straightforward from previous results.
Indeed, assertion~(i) follows from Theorem~\ref{TH_CarODETheory}\,-(iv),
assertion (ii) holds by the very definition of $w_s^*(z,t)$, and (iii) is a
consequence of Theorem~\ref{TH_CarODETheory}\,-(i).

We are left with assertion~(iv). Let $K\subset\mathcal D$ is compact. For each
$T\in E^s$ the set $\widetilde K(T):=\{\big(\varphi_{s,t}(z),\,t\big):\,
(z,s)\in K,\, s\le t\le T\}\subset\mathcal D$ is also compact, since
$(z,s,t)\mapsto\varphi_{s,t}(z)$ is continuous by
Theorem~\ref{TH_CarODETheory}\,-(v).

Denote $s_0:=\min\pr_{\mathbb R}(K)$. Then, by condition WHVF3 of
Definition~\ref{D_WHVF}, there exists a non-negative function $k_{\widetilde
K(T)}\in L^d\big([s_0,T],\Real\big)$ such that $|G(w,t)|\le k_{\widetilde
K(T)}(t)$ for any $(w,t)\in \widetilde K(T)$. Extend $k_{\widetilde K(T)}$ to
$E$ by setting $k_{\widetilde K(T)}(t)=0$ for all $t\in E\setminus[s_0,T]$.
Take any non-decreasing sequence $(T_n)\subset E$ such that $T_1>s_0$ and
$T_n\to \sup E$ as $n\to+\infty$. Define $\tilde
k_K:=\chi_{[s_0,T_1]}k_{\widetilde
K(T_1)}+\sum_{n=2}^{+\infty}\chi_{(T_{n-1},T_n]}k_{\widetilde K(T_n)}$, where
$\chi_A$ stands for the characteristic function of a set $A$. Obviously,
$\tilde k_K\in L^d_{\rm loc}(E,\Real)$ and
$\big|G\big(\varphi_{s,t}(z),t\big)\big|\le \tilde k_K(t)$ for all $(z,s)\in K$
and all $t\in E^s$.

To complete the proof it remains to recall that
$%
\varphi_{s,t}(z)-\varphi_{s,u}(z)=\int_{u}^{t}G\big(\varphi_{s,\xi}(z),\xi\big)\,d\xi
$ %
for any $(z,s)\in\mathcal D$ and any $u,t\in E^s$.
\end{proof}

The converse of Proposition~\ref{PR_sol_semicomplete} is also true: the
properties (i)--(iv) turn out to be characteristic. The exact formulation of
this fact uses the following two notions from the analysis of infinitely
dimensional vector-functions of a real variable.

Let $U$ and $W$ be some domains in the complex plane~$\Complex$ and $I\subset\Real$ an interval
containing at least two different points. By $\Hol(U,W)$ we denote the set of all holomorphic
functions from $U$ to $W$ endowed with the topology of the uniform convergence.
\begin{definition}\label{D_abs_cont}
A mapping $\varphi:I\to\Hol(U,W)$ will be called {\it absolutely continuous} on $I$ if for any
$\varepsilon>0$ and any compact set $K\subset U$ there exists $\delta=\delta(\varepsilon,K)>0$ such
that for any finite set of pairwise disjoint intervals $(I_j)_{j=1}^n$, $I_j:=(a_j,b_j)$,
$a_j,b_j\in I$, we have
$$\sum_{j=1}^n(b_j-a_j)<\delta\quad\Longrightarrow\quad \sum_{j=1}^n\|\varphi(b_j)-\varphi(a_j)\|_K<\varepsilon,$$
where $\|\cdot\|_K$ stands for the Chebyshov norm on $K$, i.e., $\|f\|_K:=\sup_{z\in K}|f(z)|$ for
any bounded function $f:K\to\Complex$.

A mapping $\varphi:I\to\Hol(U,W)$ will be called {\it locally absolutely continuous} if $\varphi$
is
absolutely continuous on each compact interval $J\subset I$.%
\end{definition}
\begin{definition}\label{D_diff}
A mapping $\varphi:I\to\Hol(U,W)$ will be called {\it differentiable} at a point $t_0\in I$ if
there exists a function $g\in\Hol(U,\Complex)$ such that
$\big(\varphi(t)-\varphi(t_0)\big)/(t-t_0)\to g$ in the topology of $\Hol(U,\Complex)$, i.e.
$$\big\|g-\big(\varphi(t)-\varphi(t_0)\big)/(t-t_0)\big\|_K\to0\quad\text{ for any
$K\subset\subset U$,}$$ as $t\to t_0$, $t\in I\setminus\{t_0\}$. The function $g$ is the {\it
derivative} of $\varphi$ at the point~$t_0$ and will be denoted by $(d\varphi/dt)(t_0)$ or $\dot
\varphi(t_0)$.
\end{definition}

\begin{remark}\label{R_R}
Let $G:U_1\to U$ and $F:W\to W_1$ be two arbitrary holomorphic functions.
Define the mapping ${C_{F,G}:\Hol(U,W)\to\Hol(U_1,W_1)}$ by setting
$C_{F,G}(g):=F\circ g\circ G$ for all $g\in\Hol(U,W)$. It is easy to show that
$C_{F,G}\circ \varphi$ is locally absolutely continuous provided so is
$\varphi:I\to\Hol(U,W)$. The same is true for the notion of differentiability:
if $\varphi:I\to\Hol(U,W)$ is differentiable at some point~$t_0\in E$, then
$C_{F,G}\circ \varphi$ is also differentiable at~$t_0$ and
$\big(d\,(C_{F,G}\circ\varphi)/dt\big)(t_0)=\big(F'\circ\varphi(t_0)\circ
G\big)\big((d\varphi/dt)(t_0)\circ G\big)$.
\end{remark}

Now we can state the following
\begin{theorem}\label{TH_sol_semicomplete}
Let $d\in[1,+\infty]$ and let $(\varphi_{s,t})_{s\in E,\,t\in E^s}$ be a family of maps
$\varphi_{s,t}:D_s\to D_t$ such that assertions~(i)--(iv) of Proposition~\ref{PR_sol_semicomplete}
hold. Then there exist a null-set $N\subset E$ and a semicomplete weak holomorphic vector field
$G:\mathcal D\to\Complex$ of order~$d$ such that the following statements are true:
\begin{mylist}
\item[(a)] for any $s\in E$ the map $E^s\ni t\mapsto \varphi_{s,t}\in\Hol(D_s,\Complex)$ is locally
absolutely continuous;
\item[(b)] for any $s\in E':=E\setminus\{\sup E\}$ the map $E^s\ni t\mapsto
\varphi_{s,t}\in\Hol(D_s,\Complex)$ is differentiable in $E^s\setminus N$;
\item[(c)] $d\varphi_{s,t}/dt=G(\cdot,t)\circ\varphi_{s,t}$ for all $s\in E'$ and all $t\in E^s\setminus
N$.
\end{mylist}
\end{theorem}

The proof of Theorem~\ref{TH_sol_semicomplete} is preceded by the following three lemmas.

\begin{lemma}\label{LM_non-const}
Under the conditions of Theorem~\ref{TH_sol_semicomplete}, for each $s\in E$ and each $t\in E^s$,
the function $\varphi_{s,t}$ is not constant in~$D_s$.
\end{lemma}
\begin{proof}
Assume on the contrary that $\varphi_{s_0,t_0}$ is constant for some $s_0,t_0\in E$ with $s_0\le
t_0$. Take $u_0:=(s_0+t_0)/2$. Then $\varphi_{s_0,t_0}=\varphi_{u_0,t_0}\circ\varphi_{s_0,u_0}$
by~(iii). Hence either $\varphi_{s_0,u_0}$ is constant in $D_{s_0}$ or $\varphi_{u_0,t_0}$ is
constant in $D_{u_0}$. Repeating this procedure several times if necessary, we see that one could
assume from the very beginning that there exist $R>0$ and $z_0\in\Complex$ such that
$K:=\overline{D(z_0,R)}\times[s_0,t_0]\subset\mathcal D$.

Combining (ii) and (iv), we see that there exists non-negative $\tilde k_K\in L^d_{\rm
loc}\big(E,\Real\big)$ such that $|\varphi_{s,t}(z)-z|\le\int_{s}^{t}\tilde k_K(\xi)d\xi$ whenever
$|z-z_0|\le R$ and $s_0\le s\le t\le t_0$. Using again bisections of the interval, we conclude that
there exist $s,t\in[s_0,t_0]$ such that $s\le t$, $\varphi_{s,t}\equiv\const$, and
$\int_{s}^{t}\tilde k_K(\xi)d\xi<R$. Then $|\varphi_{s,t}(z)-z|<R$ for all $z\in
\overline{D(z_0,R)}$, which implies that $\varphi_{s,t}(z_0+R)\neq\varphi_{s,t}(z_0-R)$. The
contradiction finishes the proof.
\end{proof}

\begin{lemma}\label{LM_toODE_contin}
Let $z_0\in\Complex$, $t_0\in E$, and $E_{z_0}\cap(-\infty,t_0]\neq\emptyset$. Then under the
conditions of Theorem~\ref{TH_sol_semicomplete}, the map $E_{z_0}\cap(-\infty,t_0]\ni
s\mapsto\varphi_{s,t_0}(z_0)$ is continuous.
\end{lemma}
\begin{proof}
Fix an arbitrary $(z_0,s_0)\in\mathcal D$ and $t_0\in E^{s_0}$. We have to prove that
$\varphi_{s,t_0}(z_0)\to\varphi_{s_0,t_0}(z_0)$ as $s\to s_0$. The proof is in two steps.

\step1 First we prove continuity from the left. So assume that $s<s_0$ and that $s\in E_{z_0}$.
From~(ii) and~(iv) it follows that $\varphi_{s,s_0}(z_0)\to z_0$ as $s\to s_0-0$. Hence by (iii),
$\varphi_{s,t_0}(z_0)=\varphi_{s_0,t_0}\big(\varphi_{s,s_0}(z_0)\big)\to\varphi_{s_0,t_0}(z_0)$ as
$s\to s_0-0$.
\endstep
\step2 If $E\ni\sup E$ and $s_0=\sup E$, then the proof is finished. So we may suppose that
$s_0<s<\sup E$. Choose any $s_1>s_0$ and $R>0$ such that
$K:=\overline{D(z_0,R)}\times[s_0,s_1]\subset\mathcal D$. Then it follows from~(ii) and (iv) that
$\varphi_{s_0,s}(z)\to z$ uniformly in $\overline{D(z_0,R)}$ as $s\to s_0+0$, $s<s_1$. Using
Rouche's theorem one can easily show that for all $s\in(s_0,s_1)$ close to $s$ enough there exists
a function $\psi_{s_0,s}:D(z_0,R/2)\to\Complex$ such that
$\varphi_{s_0,s}\circ\psi_{s_0,s}=\id_{D(z_0,R/2)}$ and that $\psi_{s_0,s}(z_0)\to z_0$ as $s\to
s_0+0$. Then by (iii), $\varphi_{s,t_0}(z_0)=\varphi_{s_0,t_0}\big(\psi_{s_0,s}(z_0)\big)\to
\varphi_{s_0,t_0}(z_0)$ as $s\to s_0+0$.
\endstep
The proof is now finished.
\end{proof}

\begin{lemma}\label{LM_toODE_contin2}
Let $h>0$, $z_0\in\Complex$, and $\tilde E_{z_0}:=E_{z_0}\cap \tilde E\neq\emptyset$, where
${\tilde E:=\{s\in E:s+h\in E\}}$. Then under the conditions of Theorem~\ref{TH_sol_semicomplete},
the map $\tilde E_{z_0}\ni s\mapsto\varphi_{s,s+h}(z_0)$ is continuous.
\end{lemma}
\begin{proof}
Fix arbitrary $(z_0,s_0)\in\mathcal D$ with $s_0\in\tilde E$. Since $\mathcal D$ is relatively open
in $\Complex\times E$, there exists $\delta\in(0,h/2)$ such that $\{z_0\}\times I$, where
$I:=[s_0-\delta,s_0+\delta]\cap \tilde E$, is a compact subset of $\mathcal D$. We have to prove
that $|\varphi_{s,s+h}(z_0)-\varphi_{s_0,s_0+h}(z_0)|\to0$ as $s\to s_0$, $s\in I$.

Note that $s_0\in I$. Moreover, for any $s\in I$ we have $s\le s_0+h$. Hence by (iv) with
$K:=\{z_0\}\times I$, we have $$\big|\varphi_{s,s+h}(z_0)-\varphi_{s,s_0+h}(z_0)\big|\le
\left|\int_{s_0+h}^{s+h}\tilde k_K(\xi)\,d\xi\right|,$$ where $\tilde k_K\in L^d_{\rm
loc}(E,[0,+\infty))$. Hence $|\varphi_{s,s+h}(z_0)-\varphi_{s,s_0+h}(z_0)|\to0$ as $s\to s_0$,
$s\in I$. To complete the proof it only remains to apply Lemma~\ref{LM_toODE_contin} with
$t_0:=s_0+h$.
\end{proof}

\begin{proof}[\bf Proof of Theorem~\ref{TH_sol_semicomplete}] First of all we notice that
statement~(a) follows directly from assertion~(iv) applied with $K:=K'\times\{s\}$, where $s\in E$
and $K'$ is an arbitrary compact subset of~$D_s$.

Let us prove (b). Consider any countable expanding system~$(K_n)$ of compact
sets ${K_n\subset \mathcal D}$ such that
$\cup_{n\in\Natural}\mathop{\mathrm{int}}K_n=\mathcal D$, where
$\mathop{\mathrm{int}} A$ stands for the interior of a set~$A$. (Such a system
exists in any locally compact separable metric space.) To simplify the notation
we will denote by $k_n$ the function $\tilde k_{K_n}$ from assertion~(iv).

Fix any $t\in E$ and any $n\in\Natural$. If $k_n(t)<+\infty$ and $t$ is a
Lebesgue point of~$k_n$, then
\begin{equation}\label{EQ_ODE_star}
C_n(t):=\sup\left\{\left|\frac{1}{t'-t}\int_{t'}^{t}k_n(\xi)\,d\xi\right|:\,t'\in
E,\,0<|t'-t|<1\right\}<+\infty.
\end{equation}
Hence~\eqref{EQ_ODE_star} holds for all $t\in E$ aside some null-set~$M_n\subset E$ containing all
non-Lebesgue points of~$k_n$. Set $M:=\cup_{n\in\Natural} M_n$.

Fix any $s\in E':=E\setminus\{\sup E\}$. Let $K\subset D_{s}$ be a compact set.
Then, by construction, there exists $n\in\Natural$ such that
$K\times\{s\}\subset K_n$. Therefore, by~\eqref{EQ_ODE_star},
\begin{equation}\label{EQ_ODE_exclam}
\left\|\frac{\varphi_{s,t'}-\varphi_{s,t}}{t'-t}\right\|_K\le
\left|\frac{1}{t'-t}\int_{t'}^{t}k_n(\xi)\,d\xi\right|\le C_n(t)<+\infty
\end{equation}
for all $t\in E^s\setminus M$ and all $t'\in E^s$ such that $0<|t'-t|<1$.

It follows from~(a) that $E^s\ni t\mapsto \varphi_{s,t}(z)$ is locally
absolutely continuous for any ${(z,s)\in\mathcal D}$. Hence there exists a
null-set $N(z,s)\subset E^s$ such that $d\varphi_{s,t}(z)/dt$ exists for all
$t\in E^s\setminus N(z,s)$. Fix an arbitrary $s\in E'$ and take any sequence
$(z_k)\subset D_s$ with at least one accumulation point in~$D_s$. Set
$N(s):=\left[\cup_{k\in\Natural}N(z_k,s)\right]\cup M$. Then, owing
to~\eqref{EQ_ODE_exclam}, Vitali's principle implies that for any~$t\in
E^s\setminus N(s)$, there exists a finite limit $\lim_{t'\to
t}(\varphi_{s,t'}-\varphi_{s,t})/(t'-t)$ in the topology
of~$\Hol(D_s,\Complex)$, i.e. $t\mapsto \varphi_{s,t}\in\Hol(D_s,\Complex)$ is
differentiable aside~$N(s)$.

Now take any $s_1\le s$, $s_1\in E$. We claim that $t\mapsto \varphi_{s,t}\in\Hol(D_s,\Complex)$ is
also differentiable for all $t\in E^s\setminus N(s_1)$. Indeed, by the above argument, $t\mapsto
\varphi_{s_1,t}$ is differentiable for all $t\in E^{s_1}\setminus N(s_1)\supset E^{s}\setminus
N(s_1)$. Since $\varphi_{s_1,t}=\varphi_{s,t}\circ\varphi_{s_1,s}$ for all $t\in E^s$, we can
conclude that $t\mapsto\varphi_{s,t}|_U$, where $U:=\varphi_{s_1,s}(D_{s_1})$, is differentiable
for all $t\in E^s\setminus N(s_1)$. By Lemma~\ref{LM_non-const}, $U\subset D_s$ is a domain. Hence
using again~\eqref{EQ_ODE_exclam} and Vitali's principle, we conclude that $t\mapsto \varphi_{s,t}$
is differentiable in $t\in E\setminus N(s_1)$.

If $E\ni\inf E$, we set $s_1:=\min E$ and $N':=N(s_1)$. If $E\not\ni\inf E$, we take any decreasing
sequence $(s_n)\subset E$ such that $s_n\to\inf E$ as $n\to+\infty$ and define
$N':=\cup_{n\in\Natural}N(s_n)$.

Finally, for the case $E\not\ni\sup E$ we set $N:=N'$. Otherwise, we set $N:=N'\cup \{\max E\}$.
Now the proof of~(b) is finished.

We are left with statement~(c). Define the function $G:\mathcal D\to\Complex$
in the following way: $G(\cdot,t):=d\varphi_{s,t}/dt|_{s:=t}$ for all $t\in
E\setminus N$, and $G(\cdot,t)\equiv0$ for all $t\in N$. Then, from
assertion~(iii), it immediately follows
$d\varphi_{s,t}/dt=G(\cdot,t)\circ\varphi_{s,t}$ for all $t\in E\setminus N$
and all $s\in E\cap(-\infty,t]$. So it remains to show that $G$ is a
semicomplete weak holomorphic field of order $d$.

First of all, $G(\cdot,t)\in\Hol(D_t,\Complex)$ for all $t\in E$ by construction. Further, for any
fixed $z\in\Complex$ such that $E_z\neq\emptyset$, we have $n\big(\varphi_{s,s+1/n}(z)-z)\to
G(z,s)$ as $n\to+\infty$ for a.e. $s\in E_z$. By Lemma~\ref{LM_toODE_contin2}, for each fixed
$n\in\Natural$ the function $s\mapsto\varphi_{s,s+1/n}$ is continuous. It follows that $G(z,\cdot)$
is measurable on~$E_z$. Therefore, $G$ satisfies conditions WHVF1 and WHVF2 from
Definition~\ref{D_WHVF}

Let us check condition WHVF3. Fix any compact set $K\subset\mathcal D$. By
construction, there exists $n\in\Natural$ such that $K\subset K_n$. Let
$(z,t)\in K$. If $t\in N$, then trivially we have $0=|G(z,t)|\le k_n(t)$. So
assume that $t\not\in N$. By construction, $t$ is a Lebesgue point for~$k_n$.
Hence, on the one hand,
$$%
Q(t,t'):=\left|\frac1{t'-t}\int_t^{t'}k_n(\xi)\,d\xi\right|\to k_n(t)\quad\text{as $t'\to t$.}
$$%
On the other hand, $G(z,t)=\lim_{t'\to t+0}\big(\varphi_{t,t'}(z)-z\big)/(t'-t)$. Passing in the
inequality $|\varphi_{t,t'}(z)-z|/|t'-t|\le Q(t,t')$ to the limit as $t'\to t+0$, we again obtain
$|G(z,t)|\le k_n(t)$.

Note that $I(K):=\pr_{\mathbb R}(K)$ is compact. Thus to finish the proof of
WHVF3, it is now sufficient to set $k_K:=k_n|_{I(K)}$.

Finally, by the above arguments, for any $(z,s)\in\mathcal D$, $E^s\ni t\mapsto\varphi_{s,t}(z)$
solves the initial value problem~\eqref{EQ_CarODE-IVP}. Thus the vector field $G$ is semicomplete.
This finishes the proof.
\end{proof}

\section{Evolution families in simply connected domains}
\label{S_evolSimply}

The notion of evolution family is one of the three central notions in modern
Loewner Theory in simply connected domains. In this paper we will use the
following definition of evolution family in a domain of the complex plane,
which for the case of the unit disk~$\UD$ was formulated in~\cite{BCM1}.

\begin{definition}\label{D_EV-simply}
Let $D$ be a domain in~$\Complex$ and $d\in[1,+\infty]$. A family
$(\varphi_{s,t})_{0\le s\le t<+\infty}$ of holomorphic self-maps
$\varphi_{s,t}:D\to D$ is said to be an {\it evolution family} of
order $d$  (or, in short, {\it $L^d$-evolution family}) in the
domain~$D$ if it satisfies the following three conditions:
\begin{mylist}
\item[EF1.] $\varphi_{s,s}=\mathsf{id}_{D},$

\item[EF2.] $\varphi_{s,t}=\varphi_{u,t}\circ\varphi_{s,u}$ whenever $0\le
s\le u\le t<+\infty,$

\item[EF3.] for any $T>0$ and any $z\in D$ there exists a
non-negative function ${k_{z,T}\in
L^{d}\big([0,T],\mathbb{R}\big)}$ such that
\[
\big|\varphi_{s,t}(z)-\varphi_{s,u}(z)\big|\leq\int_{u}^{t}k_{z,T}(\xi)d\xi
\]
whenever $0\leq s\leq u\leq t\leq T.$
\end{mylist}
\end{definition}

One can show that all non-trivial cases of evolution families {\it in a (fixed) domain} can be
reduced in one or another way to the case when $D=\UD$. We will not go into details, except for
proving the following result, which will be applied further in the paper.

\begin{proposition}\label{PR_conf-inv}
Let $d\in[1,+\infty]$. Let $(\varphi_{s,t})_{0\le s\le t}$ be a
family of holomorphic self-maps of the unit disk~$\UD$ and
$F:\UD\to\Complex$ any holomorphic univalent function. Then the
formula $\Phi_{s,t}=F\circ\varphi_{s,t}\circ F^{-1}$, $t\ge0$,
$s\in[0,t]$, defines an $L^d$-evolution family in~$D:=F(\UD)$ if
and only if $(\varphi_{s,t})$ is an $L^d$-evolution family
in~$\UD$.
\end{proposition}
The proof is based on following
\begin{lemma}\label{LM_contD}
Let $(\Phi_{s,t})$ be an $L^d$-evolution family in a hyperbolic
domain~$D$ (i.e. admitting a hyperbolic metric). Then the mapping
$\{(s,t):0\le s\le t\le T\}\ni(s,t)\mapsto \Phi_{s,t}\in\Hol(D,D)$
is continuous.
\end{lemma}
The above lemma was proved for the case $D:=\UD$ in \cite[Proposition 3.5]{BCM1}. The same argument
with obvious modifications works for any hyperbolic domain~$D$. Therefore we omit the proof.

We will also take advantage of the following well-known statement.
\begin{lemma}\label{LM_well-known}
Let $U\subset\Complex$ be a domain and $F:U\to\Complex$ a holomorphic function.
Then for any compact set $K\subset
U$ there exists a constant $C(K)>0$ such that\\
$$|F(z_1)-F(z_2)|\le C(K)|z_1-z_2|\quad \text{whenever $z_1,z_2\in K.$}$$
\end{lemma}
\begin{proof}
The lemma follows immediately from the fact that the function
$R(z_1,z_2):=|F(z_1)-F(z_2)|/|z_1-z_2|$, $z_1,z_2\in U$, $z_1\neq z_2$, can be continuously
extended to $U\times U$ by setting $R(z,z):=|F'(z)|$ for all $z\in U$.
\end{proof}

\begin{proof}[\bf Proof of Proposition~\ref{PR_conf-inv}]
Let $F$ be any of the conformal mappings of $\UD$ onto $D$.

Let us assume first that $(\varphi_{s,t})$ is an $L^d$-evolution family in~$\UD$. We have to prove
that $(\Phi_{s,t})$ is an $L^d$-evolution family in~$D$. It easy to see that $(\Phi_{s,t})$
satisfies conditions EF1 and EF2. We have to prove only EF3.

By Lemma~\ref{LM_contD} the mapping $(s,t)\mapsto\varphi_{s,t}\in\Hol(\UD,\UD)$
is continuous. Hence the set $K_{z,T}:=\{\varphi_{s,t}(z):0\le s\le t\le
T\}\subset\UD$ is compact for any fixed $z\in\UD$ and $T>0$.  Therefore, by
Lemma~\ref{LM_well-known} with $U:=\UD$, there exists $C=C(K_{z,T})>0$ such
that
$$\big|\Phi_{s,t}(w)-\Phi_{s,u}(w)\big|\le C(K_{z,T})\big|\varphi_{s,t}(z)-\varphi_{s,u}(z)\big|,\quad w:=F(z),$$
for all $s$, $u$, and $t$ satisfying $0\le s\le u\le t\le T$. Now the fact that
$(\Phi_{s,t})$ satisfies EF3 follows immediately.

The converse statement, i.e. the fact that $(\varphi_{s,t})$ is an
$L^d$-evolution family provided so is $(\Phi_{s,t})$, can be
proved in a symmetric way.
\end{proof}

The authors of~\cite{BCM1} established a deep relationship between evolution
families in the unit disk~$\UD$ and Carath\'eodory non-autonomous differential
equations, driven by the so-called {\it Herglotz vector fields}.
\begin{definition}\label{D_HVF}
A {\it Herglotz vector field} of order~$d$ (in short an {\it $L^d$-Herglotz vector field}) in a
simply connected domain~$D\varsubsetneq\Complex$  is a weak holomorphic vector field $G$ of
order~$d$ in $\mathcal D:=D\times[0,+\infty)$ such that for a.\,e. fixed $t\ge0$ the function
$G(\cdot, t)$ is an infinitesimal generator in $D$.
\end{definition}

\begin{remark}\label{RM_conf-inv-HVF}
For the notion of infinitesimal generator and the theory behind it we refer the reader
to~\cite[\S1.4.1]{Abate} or~\cite[Section~2]{BCM1}. According to the Berkson\,--\,Porta
representation~\cite{Berkson-Porta}, the set of all infinitesimal generators in~$\UD$ conincides
with the set of all functions given by the formula $G(z)=(\tau-z)(1-\bar\tau z)p(z)$, $z\in\UD$,
where $\tau$ is an arbitrary point in the closed unit disk~$\overline\UD$ and $p:\UD\to \Complex$
is an arbitrary holomorphic function satisfying the inequality~$\Re p(z)\ge0$ for all $z\in\UD$.
Moreover, if $F:\UD\to\Complex$ is a holomorphic univalent function, then the infinitesimal
generators in~$D:=F(\UD)$ are exactly the functions~$G_D$ given by the formula
$G_D(w)=F'\big(F^{-1}(w)\big)G\big(F^{-1}(w)\big)$, $w\in D$, where $G$ is an arbitrary
infinitesimal generator in~$\UD$. The non-autonomous version of this statement follows easily:
given $d\in[1,+\infty]$, the function $G:\UD\times[0,+\infty)$ is a Herglotz vector field of
order~$d$ in~$\UD$ if and only if the formula
$G_D(w,t):=F'\big(F^{-1}(w)\big)G\big(F^{-1}(w),t\big)$, $w\in D$, $t\ge0$, defines a Herglotz
vector field $G_D$ of order~$d$ in the domain~$D$.
\end{remark}
\begin{remark}\label{RM_infDW}
One of the immediate consequences of the above remark is that an infinitesimal generator
$G_D:D\to\Complex$ can have at most one zero in~$D$ unless it vanishes identically.
\end{remark}

In~\cite{BCM1} it was proved that {\it every Herglotz vector field in $\UD$ is
semicomplete} \cite[Theorem 4.4]{BCM1}. Moreover, there is a one-to-one
correspondence between these vector fields and evolution families. Namely, {\it
for every Herglotz vector field $G:\UD\times[0,+\infty)\to\Complex$ of
order~$d$ there exists a unique $L^d$-evolution family~$(\varphi_{s,t})$
generated by the vector field $G$ in the following sense: for any $z\in\UD$ and
any $s\ge0$, the function $[s,+\infty)\ni t\mapsto \varphi_{s,t}(z)$ solves the
initial value problem $\dot w=G(w,t)$, $w(s)=z$}~\cite[Theorem~5.2]{BCM1}.
Conversely, {\it every $L^d$-evolution family~$(\varphi_{s,t})$ in~$\UD$ is
generated by a unique (up to a null-set) Herglotz vector field
$G:\UD\times[0,+\infty)\to\Complex$ of order $d$}~\cite[Theorem~6.2]{BCM1}.
Using Proposition~\ref{PR_conf-inv} and Remark~\ref{RM_conf-inv-HVF} one can
easily conclude that these results hold also with~$\UD$ replaced by any simply
connected domain~$D\varsubsetneq\Complex$. Hence, taking into account
Proposition~\ref{PR_sol_semicomplete} with $\UD\times[0,+\infty)$ substituted
for~$\mathcal D$, we can state Theorem~6.2 from~\cite{BCM1} in the following, a
bit stronger form:
\begin{theorem}\label{TH_SC-EF-Diff}
Let $(\varphi_{s,t})$ be an $L^d$-evolution family in a simply
connected domain~$D\varsubsetneq\Complex$. The following
statements hold:
\begin{mylist}
\item[(i)] for any $s\ge 0$ the mapping $[s,+\infty)\ni t\mapsto
\varphi_{s,t}\in\Hol(D,D)$ is locally absolutely continuous;

\item[(ii)] moreover, the above assertion (i) holds locally uniformly w.r.t. $s$, i.e. for any $T>0$ and any  $K\subset\subset D$
there exists a non-negative function $k_{K,T}\in L^d\big([0,T],\Real\big)$, not depending on $s$,
such that
$$\|\varphi_{s,t}-\varphi_{s,u}\|_K\le\int_{u}^{t}k_{K,T}(\xi)d\xi$$ for any $s,u,t\in[0,T]$ such
that $s\le u\le t$.

\item[(iii)] there exists a Herglotz vector field $G:[0,+\infty)\to\Hol(D,\Complex)$; $t\mapsto G_t$ of order
$d$ and a null-set $N\subset[0,+\infty)$ such that for any $s\ge0$ the mapping $[s,+\infty)\ni
t\mapsto \varphi_{s,t}\in\Hol(D,\Complex)$ is differentiable aside~$N$ and for each
$t\in[s,+\infty)\setminus N$ we have $(d/dt)\varphi_{s,t}=G_t\circ\varphi_{s,t}$.
\end{mylist}
\end{theorem}

We finish this section with the proof of a modification of \cite[Proposition 2.10]{SMP}, which does
not seem to appear in the literature earlier.

\begin{proposition}\label{PR_2trajectories}
Let $d\in[1,+\infty]$, $D\varsubsetneq\Complex$ be a simply
connected domain and $(\Phi_{s,t})_{0\le s\le
t<+\infty}\subset\Hol(D,D)$. Suppose that $(\Phi_{s,t})$ satisfies
conditions EF1 and EF2 from Definition~\ref{D_EV-simply}. Let
$\zeta_1,\zeta_2\in D$ and $\zeta_1\neq\zeta_2$. If the functions
$t\mapsto \Phi_{0,t}(\zeta_1)$ and $t\mapsto \Phi_{0,t}(\zeta_2)$
belong to $AC^d\big([0,+\infty),D\big)$ and
$\Phi_{0,t}(\zeta_1)\neq \Phi_{0,t}(\zeta_2)$ for all $t\ge0$,
then $(\Phi_{s,t})$ is an evolution family of order~$d$ in the
domain~$D$.
\end{proposition}
To prove the above proposition we need following
\begin{lemma}\label{LM_Santiago}
There exists a universal constant $C>0$ such that for any holomorphic map $\varphi:\UD\to\UD$  with
$\varphi (0)=0,$ any $r\in (0,1)$
and any $\zeta_{0}\in \mathbb{D\setminus \{}0\}$, the following inequality holds:%
\begin{equation}\label{EQ_SantPav}
|\varphi (\zeta)-\zeta|\leq \frac{C}{|\zeta_{0}|(1-|\zeta_{0}|^{2})(1-r^{2})}|\varphi
(\zeta_{0})-\zeta_{0}|,\text{ }\qquad |\zeta|\leq r.
\end{equation}
\end{lemma}
\begin{proof}
We fix arbitrary $\zeta_{0}\in \mathbb{D\setminus \{}0\}$, $r\in (0,1)$, and $\varphi \in
\mathrm{Hol}(\mathbb{D},\mathbb{D})$  with $\varphi (0)=0$. If $|\varphi (\zeta_{0})-\zeta_{0}|\geq
|\zeta_{0}|$, then~\eqref{EQ_SantPav} holds trivially with any $C\ge2$.

So we can assume that $|\varphi (\zeta_{0})-\zeta_{0}|<|\zeta_{0}|.$ Then $|\varphi
(\zeta_{0})+\zeta_{0}|> |\zeta_{0}|>0$. Hence, it follows from the Schwarz lemma that
$\varphi'(0)\neq-1$ and that $\varphi(\zeta)/\zeta\in\overline{\mathbb D}\setminus\{-1\}$ for all
$\zeta\in\mathbb D\setminus\{0\}$.

Therefore, the function
$p(\zeta):=\big(1-\varphi(\zeta)/\zeta\big)/\big(1+\varphi(\zeta)/\zeta\big)$, extended to the
origin by continuity, is holomorphic in $\mathbb D$ and satisfies there the inequality $\func{Re}
p(\zeta)\ge 0$. Hence,~\cite[ineq.\,(11) on p.\,40]{Pommerenke} implies that for
any $|\zeta|\leq r,$%
\begin{multline*}
|p(\zeta)| \leq |\func{Im} p(\zeta_0)|+|\func{Re} p(\zeta_{0})|\,\frac{1+\left\vert \dfrac{\zeta_{0}-\zeta}{1-\overline{\zeta_{0}}%
\zeta}\right\vert }{1-\left\vert \dfrac{\zeta_{0}-\zeta}{1-\overline{\zeta_{0}}\zeta}\right\vert
}\le
\sqrt2\,|p(\zeta_{0})|\,\frac{1+\left\vert \dfrac{\zeta_{0}-\zeta}{1-\overline{\zeta_{0}}%
\zeta}\right\vert }{1-\left\vert \dfrac{\zeta_{0}-\zeta}{1-\overline{\zeta_{0}}\zeta}\right\vert}=
\\
\vphantom{\int\limits_0^0}=\sqrt2\,|p(\zeta_{0})|\,\frac{(|1-\overline{\zeta_{0}}\zeta|+|\zeta_{0}-\zeta|)^2}{%
(1-|\zeta_{0}|^{2})(1-|\zeta|^{2})} \leq |p(\zeta_0)|\frac{16\sqrt2}{(1-|\zeta_{0}|^{2})(1-r^{2})}.
\end{multline*}%
Finally, bearing in mind that $|\varphi (\zeta_{0})+\zeta_{0}|> |\zeta_{0}|$ we deduce that,
whenever $0<|\zeta|\leq r$,
\begin{eqnarray*}
|\varphi (\zeta)-\zeta| &=&|\zeta+\varphi (\zeta)|\,|p(\zeta)|\leq \frac{32\sqrt2}{%
(1-|\zeta_{0}|^{2})(1-r^{2})}\frac{|\zeta_{0}-\varphi (\zeta_{0})|}{|\zeta_{0}+\varphi
(\zeta_{0})|} \\
&\leq &\frac{32\sqrt2}{|\zeta_{0}|(1-|\zeta_{0}|^{2})(1-r^{2})}|\varphi (\zeta_{0})-\zeta_{0}|.
\end{eqnarray*}
This finishes the proof.
\end{proof}

\begin{proof}[\bf Proof of Proposition~\ref{PR_2trajectories}] First we prove
the proposition for the case $D:=\UD$. For the sake of convenience we change
the notation used in Proposition~\ref{PR_2trajectories}. So let
$(\varphi_{s,t})_{0\le s\le t<+\infty}\subset\Hol(\UD,\UD)$, $z_1,z_2\in\UD$,
$d\in[1,+\infty]$. Assume that (1)~$\varphi_{s,s}=\id_{\mathbb D}$ for all
$s\ge0$; (2)~ $\varphi_{s,t}=\varphi_{u,t}\circ \varphi_{s,u}$ whenever $0\le
s\le u\le t<+\infty$; (3)~$t\mapsto \varphi_{0,t}(z_1)$ and $t\mapsto
\varphi_{0,t}(z_2)$ belong to $AC^d\big([0,+\infty),\mathbb D\big)$;
(4)~$\varphi_{0,t}(z_1)\neq \varphi_{0,t}(z_2)$ for all $t\ge0$. We have to
prove that $(\varphi_{s,t})$ is an $L^d$-evolution family in~$\UD$.

Write $$h_t(z):=\frac{z+a(t)}{1+\overline{a(t)}z},\quad a(t):=\varphi_{0,t}(z_1).$$ Define
$\psi_{s,t}:=h_t^{-1}\circ\varphi_{s,t}\circ h_s$. We claim that $(\psi_{s,t})$ is an
$L^d$-evolution family. Clearly, $(\psi_{s,t})$ satisfies conditions EF1 and EF2 from
Definition~\ref{D_EV-simply}. Moreover, $\psi_{s,t}(0)=0$ for any $s\ge0$ and any $t\ge s$. To
check condition EF3, we note that $t\mapsto \zeta_0(u):=\psi_{0,t}(z_0)\in \UD$, where
$z_0:=h_0^{-1}(z_2)$, belongs to~$AC^d\big([0,+\infty),\UD\big)$ and does not vanish.

Now take arbitrary $T>0$. Fix any $z\in\UD$ and any $s,u,t\in[0,T]$ such that $s\le u\le t$. Denote
$r:=|z|$ and $\varepsilon:=\min_{u\in[0,T]}|\zeta_0(u)|>0$. Applying Lemma~\ref{LM_Santiago} with
$\varphi:=\psi_{u,t}$, $\zeta:=\psi_{s,u}(z)$, $\zeta_0:=\zeta_0(u)$ and taking into account that
$|\zeta|\le|z|=r$ and $|\zeta_0(u)|\le |z_0|$ by the Schawrz lemma, we get
\begin{multline*}
|\psi_{s,t}(z)-\psi_{s,u}(z)|=|\psi_{u,t}(\zeta)-\zeta|\le
\frac{C}{|\zeta_{0}(u)|(1-|\zeta_{0}(u)|^{2})(1-r^{2})}|\psi_{u,t}(\zeta_{0})-\zeta_{0}|\le\\\le
\frac{C}{\varepsilon(1-|z_{0}|^{2})(1-r^{2})}|\psi_{0,t}(z_{0})-\psi_{0,u}(z_{0})|.
\end{multline*}
This shows that $(\psi_{s,t})$ satisfies EF3 and hence it is an $L^d$-evolution
family. Now we apply~\cite[Lemma~2.8]{SMP} to conclude that~$(\varphi_{s,t})$
is also an evolution family of order~$d$.

The proof for the case $D:=\UD$ is complete. For arbitrary simply connected
domains ${D\varsubsetneq\Complex}$, the proposition follows now from the
Riemann Mapping Theorem and Proposition~\ref{PR_conf-inv}.
\end{proof}

\section{Evolution families over systems of doubly connected
domains}\label{S_evolDoubly}
\subsection{Definition of an evolution family in doubly connected case}
As we mentioned in the introduction, the most important new property of Loewner
Evolution in multiply connected case is that the canonical  domain has to
evolve in time, while in simply connected case the conformal type does not
change. On the level of evolution families one can explain this phenomenon by
the fact that all the families satisfying Definition~\ref{D_EV-simply} for
$D:=\mathbb A_r$ with some fixed $r\in(0,1)$ are exhausted by rotations (see
Example~\ref{EX_Rotations}). So instead of one fixed reference domain we
consider families of reference domains. A natural choice of doubly connected
reference domains are the annuli $\mathbb A_r$, where $r\in[0,1)$. (Note we do
not exclude the case $r=0$.) With each annulus~$\mathbb A_r$ we can  associate
a one-generated torsion-free Fuchsian group~$\Gamma$ such that $\mathbb A_r$ is
conformally equivalent to~$\UD/\Gamma$. This group is unique up to conjugation
by a M\"obius transformation and the conjugation classes are uniquely defined
by the multiplier~$\lambda$ of the generator of $\Gamma$. It is not difficult
to calculate that $\lambda=e^{-2\pi\omega(r)}$, where
\begin{equation}\label{EQ_omega}
\omega(r):=\left\{%
\begin{array}{ll}
   -\pi/\log r,& \text{if~}r\in(0,1),\\%
   0,& \text{if~}r=0.
\end{array}\right.
\end{equation}
Hence it is natural to consider families of annuli $(\mathbb A_{r(t)})_{t\ge0}$
assuming some regularity of the function $t\mapsto\omega(r(t))$. Namely, we
introduce the following
\begin{definition}\label{def-cansys}
Let $d\in[1,+\infty]$ and $(D_t)_{t\ge0}$ be a family of annuli $D_t:=\mathbb A_{r(t)}$. We will
say that $(D_t)$ is a {\it (doubly connected) canonical domain system of order~$d$} (or in short, a
{\it canonical $L^d$-system}) if the function $t\mapsto \omega(r(t))$ belongs to
$AC^d\big([0,+\infty),[0,+\infty)\big)$ and does not increase. If $r(t)\equiv0$, then the canonical
domain system~$(D_t)$ will be called {\it degenerate}. If on the contrary $r(t)$ does not vanish,
then ~$(D_t)$ will be called {\it non-degenerate}. Finally, if there exists $T>0$ such that
$r(t)>0$ for all $t\in[0,T)$ and $r(t)=0$ for all $t\ge T$, then we will say that~$(D_t)$ is {\it
of mixed type}.
\end{definition}

\begin{remark}
The condition that $t\mapsto \omega(r(t))$ is of class $AC^d$ implies that $t\mapsto r(t)$ also
belongs to $AC^d\big([0,+\infty),[0,1)\big)$. In the non-degenerate case, i.e. when $r(t)>0$ for
all~$t\ge0$, or if $d=1$, then the converse is also true and we can replace $\omega(r(t))$ by
$r(t)$ in the above definition. However, in general this is not the case and for some auxiliary
statements (e.\,g. for Lemma~\ref{LM_lifting_evolut_fam}) the condition $\omega\circ r\in
AC^d\big([0,+\infty),[0,+\infty)\big)$ is essential even if we assume that $t\mapsto r(t)$ is of
class~$AC^d$. At the same time we do not know whether any of our main results would fail to hold in
the mixed case with $d>1$ if in Definition~\ref{def-cansys} one places a weaker condition $r\in
AC^d\big([0,+\infty),[0,1)\big)$ instead of $\omega\circ r\in
AC^d\big([0,+\infty),[0,+\infty)\big)$.
\end{remark}

Now we can introduce the definition of an evolution family for the doubly
connected setting.

\begin{definition}\label{def-ev}
Let $(D_t)_{t\ge0}$ be a canonical domain system of order~$d\in[1,+\infty]$. A family
$(\varphi_{s,t})_{0\leq s\leq t<+\infty}$ of holomorphic mappings $\varphi_{s,t}:D_s\to D_t$ is
said to be an {\it evolution family of order $d$ over~$(D_t)$} (in short, an {\it $L^d$-evolution
family}) if the following conditions are satisfied:
\begin{mylist}
\item[EF1.] $\varphi_{s,s}=\mathsf{id}_{D_s},$

\item[EF2.] $\varphi_{s,t}=\varphi_{u,t}\circ\varphi_{s,u}$ whenever $0\le
s\le u\le t<+\infty,$

\item[EF3.] for any closed interval $I:=[S,T]\subset[0,+\infty)$ and any $z\in D_S$ there exists a
non-negative function ${k_{z,I}\in
L^{d}\big([S,T],\mathbb{R}\big)}$ such that
\[
|\varphi_{s,u}(z)-\varphi_{s,t}(z)|\leq\int_{u}^{t}k_{z,I}(\xi)d\xi
\]
whenever $S\leq s\leq u\leq t\leq T.$
\end{mylist}
Suppressing the language we will refer also to the pair $\mathcal
E:=\big((D_t), (\varphi_{s,t})\big)$ as an evolution family of order~$d$ and
apply terms {\it degenerate}, {\it non-degenerate}, {\it of mixed type} to
$\mathcal E$ whenever they are applicable to the canonical domain
system~$(D_t)$.
\end{definition}

It is clear that the notion of an evolution family over a {\it degenerate} canonical domain system
given by Definition~\ref{def-ev} is the same as the notion of evolution family in the domain
${D:=\UD^*}$ given by Definition~\ref{D_EV-simply}. This case is known to be equivalent to that of
evolution families in the unit disk fixing the origin. We will discuss it briefly in
Section~\ref{S_degenerate}, while the main attention in this paper will be paid to the
non-degenerate case.

\subsection{Lifting evolution families to a simply connected domain}
Given a canonical domain system~$(D_t)$ and a family $(\varphi_{s,t})$
satisfying algebraic conditions EF1 and EF2 from Definition~\ref{def-ev}, there
is a lifting of $(\varphi_{s,t})$ to the upper half-plane $\UH:=\{z:\Im z>0\}$
(or any other hyperbolic simply connected domain) satisfying conditions EF1 and
EF2 from Definition~\ref{D_EV-simply}. Under some additional conditions such
lifting is unique. An important role in our arguments is played by the class
$\classM(r_1,r_2)$ of all functions $\psi\in\Hol(\mathbb A_{r_1},\mathbb
A_{r_2})$, $1>r_1\ge r_2\ge 0$, such that $I(\psi\circ\gamma)=I(\gamma)$ for
any oriented closed curve $\gamma\subset\mathbb A_{r_1}$, where $I(\gamma)$
stands for the index of the point $z=0$ w.r.t. $\gamma$.

The lifting technique allows us to apply the theory of evolution families in simply connected
domains to establish some useful results in the doubly connected case. One of these results is the
following analogue of~Proposition~\ref{PR_2trajectories}, which gives a sufficient (and in fact
necessary) condition for $(\varphi_{s,t})$ to be an $L^d$-evolution family over~$(D_t)$.

\begin{theorem}\label{TH_1punto} Let $\big(D_t\big)=\big(\mathbb
A_{r(t)}\big)$ be a canonical  domain system of order~$d\in[1,+\infty]$ and let
$(\varphi_{s,t})_{0\le s\le t}$ be a family of holomorphic functions
$\varphi_{s,t}:D_s\to D_t$ satisfying conditions EF1 and EF2 in
Definition~\ref{def-ev}. Suppose that at least one the following conditions
holds:
\begin{mylist}
\item[(a)] for each $t>0$, each $s_0\in[0,t)$ and any $z\in D_{s_0}$ the mapping $[s_0,t]\ni s\mapsto \varphi_{s,t}(z)\in \UD^*$ is
continuous;
\item[(b)] for each $s\ge0$ and any $z\in D_s$ the mapping $[s,+\infty)\ni t\mapsto \varphi_{s,t}(z)\in \UD^*$ is
continuous;
\item[(c)] for each $s\ge0$ and $t\ge s$ the function $\varphi_{s,t}$ belongs to the class
$\classM\big(r(s),r(t)\big)$.
\end{mylist}
If there exists a point $z_0\in D_0$ such that the function $[0,+\infty)\ni t\mapsto
\varphi_{0,t}(z_0)\in\UD^*$ belongs to $AC^d\big([0,+\infty),\UD^*\big)$, then $(\varphi_{s,t})$ is
an $L^d$-evolution family over $(D_t)$.
\end{theorem}

Using the lifting technique, we can also extend assertion~(ii) of
Theorem~\ref{TH_SC-EF-Diff} to the doubly connected setting.

\begin{proposition}\label{PR_uniformity}
Let $(\varphi_{s,t})$ be an $L^d$-evolution family over a canonical
$L^d$-system $\big(D_t\big)=\big(\mathbb A_{r(t)}\big)$. Then for any closed
interval $I:=[S,T]\subset[0,+\infty)$ and any compact set $K\subset D_{S}$
there exists a non-negative function $k_{K,I}\in L^d\big(I,\Real\big)$ such
that
$$
\|\varphi_{s,t}-\varphi_{s,u}\|_K\le\int_{u}^{t}k_{K,I}(\xi)d\xi
$$
for all $s,u,t\in I$ satisfying $s\le u\le t$.
\end{proposition}

Our study of the vector fields corresponding to non-degenerate $L^d$-evolution families is also
based on the lifting technique. A suitable geometry of the covering space for non-degenerate case
is the one of the strip~$\mathbb S:=\{z:0<\Re z<1\}$. This is a motivation for following
\begin{theorem}\label{TH_lifting_evolut_fam}
Let $\big(D_t\big)=\big(\mathbb A_{r(t)}\big)$ be a non-degenerate canonical
domain system of order~$d\in[1,+\infty]$. Then for any $L^d$-evolution family
$(\varphi_{s,t})$ over $(D_t)$ there exists a unique $L^d$-evolution family
$(\Psi_{s,t})$ in the strip~$\mathbb S$ such that
\begin{equation}\label{EQ_G-phi}
W_t\circ\Psi_{s,t}=\varphi_{s,t}\circ W_s,\quad 0\le s\le t<+\infty,
\end{equation}
where $W_\tau(\zeta):=\exp(\zeta\log r(\tau))$ for all $\tau\ge0$ and all $\zeta\in\mathbb S$.
\end{theorem}

The proofs are given in Section~\ref{S_proofs1} and based on some lemmas we are going to establish
in the next section.

\subsection{Some auxiliary statements}\label{SS_aux}

\REM{By $\mathbb U(r_1,r_2)$ we will denote the subclass consisting of all univalent
functions~$\psi\in\classM(r_1,r_2)$.}
%
%

%REM begins
\REM{Now let $\psi\in \classM(r_1,r_2)$. The function $F_\psi (z):=\log\big(\psi(z)/z\big)$ is
analytic and single-valued in~$\mathbb A_{r_1}$, because for any closed curve $\gamma\subset
\mathbb A_{r_1}$ the increment of $\log \psi(z)$ along $\gamma$ is the same as that of $\log z$.
However, $F_\psi$ is defined up to an additive constant of the form $2\pi i n$, $n\in\mathbb Z$. To
avoid this ambiguity we consider the quantity $\mathcal M(\psi):=\exp\mathcal N(F_\psi)$. Now we
define the subclasses $\classM_0(r_1,r_2)$ and $\mathbb U_0(r_1,r_2)$ of the classes
$\classM(r_1,r_2)$ and $\mathbb U(r_1,r_2)$, respectively, consisting of functions $\psi$ with
$\mathcal M(\psi)=1$.
}%end of \REM

%THESE LINES DO NOT SEEM TO  BE NECESSARY
\REM{ Note that the classes $\classM(r_1,r_2)$ and
$\classM_0(r_1,r_2)$ have a property, similar to that of
semigroups.
\begin{proposition}\label{PR_semigroup-pr}
Let $1>r_1\ge r_2\ge r_3>0$. The following statements hold:
\begin{itemize}
\item[(i)] if $\psi_1\in \classM(r_1,r_2)$ and $\psi_2\in \classM(r_2,r_3)$, then
$\psi_2\circ\psi_1\in \classM(r_1,r_3)$, with\\ ${\mathcal M(\psi_3)=\mathcal M(\psi_1)\mathcal
M(\psi_2)}$;
\item[(ii)] if $\psi_1\in \classM_0(r_1,r_2)$ and $\psi_2\in \classM_0(r_2,r_3)$, then
$\psi_2\circ\psi_1\in \classM_0(r_1,r_3)$.
\end{itemize}
\end{proposition}
[PROOF IS TO BE WRITTEN] }

\begin{lemma}\label{LM_evol-fam-classM}
Suppose $\big((D_t),(\varphi_{s,t})\big)$ is an evolution family of order $d\in[1,+\infty]$.
Let $s\ge 0$. Then the following statements are true:
\begin{mylist}
\item[(i)] for each $z\in D_s$ the function $t\mapsto \varphi_{s,t}(z)$ belongs to
$AC^d\big([s,+\infty),\Complex\big)$;
\item[(ii)] the mapping $t\mapsto\varphi_{s,t}\in \Hol(\mathbb A_{r(s)},\mathbb D^*)$, $\mathbb D^*:=\mathbb
D\setminus\{0\}$, is continuous in $[s,+\infty)$;
\item[(iii)] $\varphi_{s,t}\in \classM\big(r(s),r(t)\big)$ for any $t\ge s$;
\item[(iv)] $\varphi_{s,t}$ is univalent in $D_s$ for any $t\ge s$;
%REM begins
\REM{\item[(v)] the function $t\mapsto \mathcal M_{s,t}:=\mathcal M(\varphi_{s,t})$ is continuous
on $[s,+\infty)$, with $\mathcal M_{s,t}=\mathcal M_{0,t}/\mathcal M_{0,s}$.
%\item[(vi)] for any $S\ge0$, any $T>S$ and any compact set $K\subset D_S$ we have $$\big\{\varphi_{s,t}(z):
%z\in K,~S\le s\le t\le T\big\}\subset\subset D_T.$$
}%end of \REM
\end{mylist}
\end{lemma}
\begin{proof}
To prove (i) we need only to apply condition EF3 from Definition~\ref{def-ev}
for $S:=s$. Fix $t\ge s$. From (i) it now follows that
$\varphi_{s,u}(z)\to\varphi_{s,t}(z)$ pointwise in $D_s$ as $u\to t$. The
functions $\varphi_{s,t}$, $t\ge s$, form (for fixed~$s\ge0$) a normal family
in~$D_s$. Therefore, the pointwise convergence implies convergence of
$\varphi_{s,u}$ to $\varphi_{s,t}$ in~$\Hol(D_s,\Complex)$. This proves~(ii).

Now let us take any closed curve $\gamma:[0,1]\to D_s$. Fix $t\ge s$. Recall that
$\varphi_{s,s}=\id_{D_s}$. Hence, (ii) implies that $g(u,x):=\varphi_{s,u}\big(\gamma(x)\big)$,
$u\in[s,t]$, $x\in[0,1]$, provides us with a homotopical deformation of the curve $\gamma$ into the
curve $\varphi_{s,t}\circ\gamma$ within the domain~$D_t$. It follows that
$I(\varphi_{s,t}\circ\gamma)=I(\gamma)$. Since $\gamma$ is chosen arbitrarily, (iii) is now also
proved.

To prove (iv) we argue as in~\cite[Proposition~3]{BCM2}. Assume  that there exist $s\ge0$, $t>s$
and $z_1,z_2\in D_s$ such that $z_1\neq z_2$ but $\varphi_{s,t}(z_1)=\varphi_{s,t}(z_2)$. Denote
$\zeta_j(u):=\varphi_{s,u}(z_j)$, $j=1,2$. Let $u_0:=\inf\{u\in[s,t]:\zeta_1(u)=\zeta_2(u)\}$.
Clearly, $u_0\in[s,t]$. From (i) we know that the functions $\zeta_j$ are continuous. Therefore,
$\zeta_1(u_0)=\zeta_2(u_0):=\zeta_0$. In particular, $u_0\neq s$. At the same time, by construction
\begin{equation}\label{EQ_z1z2}
\zeta_1(u)\neq\zeta_2(u),\quad u\in[s,u_0).
\end{equation}
Let $U\ni \zeta_0$ be any domain such that $\overline U\subset D_{u_0}$. Then there exists
$u_1\in[s,u_0)$ such that $\{\zeta_1(u),\zeta_2(u)\}\subset U\subset D_u$ for all $u\in[u_1,u_0]$.
In particular, the functions $\varphi_{u,u_0}$ are well-defined and holomorphic in~$U$ for all
$u\in[u_1,u_0]$. By condition EF2 in Definition~\ref{def-ev},
\begin{equation}\label{EQ_u1u0}
\varphi_{u,u_0}\big(\zeta_1(u)\big)=\varphi_{u,u_0}\big(\zeta_2(u)\big)=\zeta_0,\quad
u\in[u_1,u_0].
\end{equation}

Now we claim that $\varphi_{u,u_0}|_U\to\id_U$ in $\Hol(U,\Complex)$ as $u\to
u_0-0$. The pointwise convergence $\varphi_{u,u_0}(z)\to z$, $z\in U$, is a
consequence of condition EF3 in Definition~\ref{def-ev} applied with
$(s,u,u,u_0,t)$ substituted for $(S,s,u,t,T)$. Since the functions
$\varphi_{u,u_0}$, $u\in[u_1,u_0]$, form a normal family in $U$, the pointwise
convergence implies convergence in~$\Hol(U,\Complex)$. In particular, it
follows that given a sufficiently small open neighborhood~$W$ of the
point~$\zeta_0$, the function $\varphi_{u,u_0}$ is univalent on $W$ provided
$u$ is sufficiently close to $u_0$. The fact that this statement contradicts
relations~\eqref{EQ_z1z2} and~\eqref{EQ_u1u0} proves assertion~(iv).
\end{proof}

\begin{lemma}\label{LM_pre_lifting_evolut_fam} Under conditions of Theorem~\ref{TH_1punto} assertion~(c) holds.
\end{lemma}
\begin{proof}
Assume first that (b) holds. Fix any $s\ge0$. By normality of the family
$(\varphi_{s,t})_{t\ge0}$ in $D_s$, assertion~(b) implies that the map
$[s,+\infty)\ni t\mapsto\varphi_{s,t}\in\Hol(D_s,\UD^*)$ is continuous. Hence
as in the proof of Lemma~\ref{LM_evol-fam-classM} we can conclude that for any
fixed $t\ge s$ the map $[0,1]\times
D_s\ni(\theta,z)\mapsto\varphi_{s,t(\theta)}(z)\in\UD^*$, where
$t(\theta):=(1-\theta)t+\theta s$, provides us with a homotopical family in
$\UD^*$ joining $\varphi_{s,t}$ with $\id_{D_s}$. It follows immediately that
$\varphi_{s,t}\in\classM\big(r(s),r(t)\big)$. This proves that
(b)$\Rightarrow$(c).

The proof of the implication~(a)$\Rightarrow$(c) is similar. Combining (a) with the normality
argument we can conclude that for each $s\ge0$ and $t\ge s$ the map $[0,1]\times
D_s\ni(\theta,z)\mapsto\varphi_{s(\theta),t}(z)\in\UD^*$, where $s(\theta):=(1-\theta)s+\theta t$,
is a homotopical family in $\UD^*$ joining $\varphi_{s,t}$ with $\id_{D_s}$. Again (c) follows
immediately.
\end{proof}
Denote by $\mathbb S_r$, $r\in(0,1)$, the strip $\{z:\,\log r<\Re z<0\}$ and by
$\mathbb S_0$ the left half-plane $\{z:\,\Re z<0\}$.
\begin{lemma}\label{LM_lifting_evolut_fam}
Under conditions of Theorem~\ref{TH_1punto} there exists an evolution family $(\Phi_{s,t})$ of
order~$d$ in $\UH$ such that for all $s\ge0$ and all $t\ge s$ we have
\begin{equation}\label{EQ_G-phi0}
B_t\circ\Phi_{s,t}=\varphi_{s,t}\circ B_s,
\end{equation}
where $B_\tau(w):=\exp Q\big(w,\omega_\tau\big)$, $\tau\ge0$, and
$$Q(w,\omega):=\frac{i}{\omega}\log\frac{1+\omega w}{1-\omega w},~Q(w,0):=2iw,~ \omega_{\tau}:=\omega(r(\tau))=\left\{\begin{array}{ll}
-\pi/\log r(\tau),& \mathrm{if~}r(\tau)>0,\\0,&\mathrm{if~}r(\tau)=0.\end{array}\right.$$
\end{lemma}
\begin{remark}\label{RM_LM}
By $\log$ in the formula for the function $Q$ in the above lemma we mean the
single-valued branch of the logarithm in $\Complex\setminus(-\infty,0]$ that
vanishes at $\zeta=1$. Hence the function
 $Q$ being extended by continuity to $\omega=0$ is well-defined and holomorphic in the
domain $\big\{(w,\omega)\in\Complex^2:\,\omega w\not\in(-\infty,-1]\cup[1,+\infty)\big\}$. For each
$\omega\ge0$ the function $Q(\cdot,\omega)$ maps $\mathbb H$ conformally onto the strip $\mathbb
S_r$, where $r:=e^{-\pi/\omega}$, if $\omega>0$, or onto the left half-plane~$\mathbb S_0$ if
$\omega=0$.
\end{remark}
\begin{proof}[Proof of Lemma~\ref{LM_lifting_evolut_fam}]
Owing to Lemma~\ref{LM_pre_lifting_evolut_fam} we can assume that assertion (c) from
Theorem~\ref{TH_lifting_evolut_fam} takes place. Let us construct first the family $(\Phi_{s,t})$
and then prove that it is an $L^d$-evolution family.

Denote by $R(\cdot,\omega)$ the function inverse to $Q(\cdot,\omega)$, i.e.
$$
R(\zeta,\omega)=\frac1\omega\,\frac{e^{-i\zeta\omega}-1}{e^{-i\zeta\omega}+1},
$$
for $\omega\neq0$, and $R(\zeta,0)=-i\zeta/2$ when $\omega=0$. Therefore $R$ is
holomorphic in
$\Complex^2\setminus\big\{(\zeta,\omega):\zeta\omega=\pi(n+1/2)~\text{for some
}n\in\mathbb Z\big\}$. Consider the curve $[0,+\infty)\ni t\mapsto
z(t):=\varphi_{0,t}(z_0)\in\UD^*$ and let $t\mapsto\zeta(t)$ be any of its
liftings  w.r.t. the covering map $\exp:\mathbb S_0\to\UD^*$. Finally, define
$w(t):=R\big(\zeta(t),\omega_t\big)$ for all $t\ge0$, where $\omega_t$ is
introduced in the statement of Lemma~\ref{LM_lifting_evolut_fam}. Then
$t\mapsto w(t)\in \mathbb H$ is continuous and $B_t(w(t))=z(t)$ for all
$t\ge0$.

Fix any $s\ge0$ and any $t\ge s$. Taking into account EF2, we have $(\varphi_{s,t}\circ
B_s)(w(s))=z(t)$. Hence $(\varphi_{s,t}\circ B_s)(w(s))=B_t(w(t))$. Note that according to
Remark~\ref{RM_LM}, $B_t:\UH\to \mathbb A_{r(t)}$ is a covering map. It follows that there exists a
unique lifting $F$ of $\varphi_{s,t}\circ B_s:\UH\to\UD^*$ w.r.t. $B_t$ which takes~$w(s)$ to
$w(t)$. Now put $\Phi_{s,t}:=F$.

The above argument defines a family $(\Phi_{s,t})_{0\le s\le t}\subset\Hol(\UH,\UH)$.
Equality~\eqref{EQ_G-phi0} takes place by construction, while conditions EF1 and EF2 for
$(\Phi_{s,t})$ follow from conditions EF1 and EF2 for $\varphi_{s,t}$ and the uniqueness of the
lifting. Now according to Theorem~\ref{PR_2trajectories} it remains to find two points
$w_1,w_2\in\mathbb H$ such that the functions $t\mapsto w_1(t):=\Phi_{0,t}(w_1)$ and $t\mapsto
w_2(t):=\Phi_{0,t}(w_2)$ are of class~$AC^d$ and $w_1(t)\neq w_2(t)$ for all $t\ge0$.

Put $w_1:=R(\zeta(0),\omega_0)$, $w_2:=R(\zeta(0)+2\pi i,\omega_0)$. Then by construction,
$w_1(t)=w(t)$ for all $t\ge0$. We claim that
\begin{equation}\label{EQ_w_2}
w_2(t)=R\big(\zeta(t)+2\pi i,\,\omega(t)\big)\quad \text{for all $t\ge0$.}
\end{equation}
The above equality is obviously equivalent to stating that $\phi_t(\zeta(0)+2\pi i)=\zeta(t)+2\pi
i$, where $\phi_t=Q(\cdot,\omega_t)\circ\Phi_{0,t}\circ R(\cdot,\omega_0)$. Note that
$\phi_t(\zeta(0))=\zeta(t)$. Hence to prove~\eqref{EQ_w_2} it is sufficient now to show that
\begin{equation}\label{EQ_phi}
\phi_t(\zeta+2\pi i)=\phi_t(\zeta)+2\pi i\quad\text{for all $\zeta\in\mathbb S_{r(0)}$ and all
$t\ge0$.}
\end{equation}
Recall that for each $s\ge0$ and $t\ge s$, the function $\Phi_{s,t}$ was constructed to be a
lifting of $\varphi_{s,t}\circ B_s:\UH\to\UD^*$ w.r.t. $B_t$. Therefore, for each $t\ge0$, the
function $\phi_t$ is the lifting of $\varphi_{0,t}\circ\exp\!|_{\mathbb S_{r(0)}}$ w.r.t.
$\exp\!|_{\mathbb S_{r(t)}}$. It follows that for each fixed $\zeta\in\mathbb S_{r(0)}$ the curve
$\phi_t\circ \gamma$, where $\gamma:[0,1]\to\mathbb S_{r(0)}; \theta\mapsto \zeta+2\pi i\theta$, is
a lifting of the curve $\gamma_1:=\varphi_{0,t}\circ\exp\circ\gamma$ w.r.t.~$\exp\!|_{\mathbb
S_{r(t)}}$. Consequently, taking into account that $\varphi_{0,t}\in\classM\big(r(0),r(t)\big)$, we
get
\begin{equation*} \phi_t(\zeta+2\pi
 i)-\phi_t(\zeta)=\int_{\gamma_1}\frac{dz}{z}=2\pi i\,I(\gamma_1)=2\pi
 i\,I(\exp\circ\gamma)=2\pi i,\\
\end{equation*}
where $I(\cdot)$ stands for the index of the origin  w.r.t. a curve.

Now it remains to show that $t\mapsto w_1(t)$ and $t\mapsto w_2(t)$ are of class $AC^d$. Recall
that $t\mapsto \zeta(t)$ is a lifting of $t\mapsto\varphi_{0,t}(z_0)$, which is of class $AC^d$.
Hence $t\mapsto \zeta(t)$ is of class $AC^d$ as well. Finally, by definition $t\mapsto\omega_t$ is
also of class $AC^d$. Therefore $w_1(t)=w(t)=R\big(\zeta(t),\,\omega_t\big)$ is locally absolutely
continuous and the derivative
$$
\frac{dw_1(t)}{dt}=R'_\zeta\big(\zeta(t),\,\omega_t\big)\frac{d\zeta(t)}{dt}+R'_\omega\big(\zeta(t),\,\omega_t\big)\frac{d\omega_t}{dt}
$$
belongs to $L^d_{\rm loc}\big([0,+\infty),\Complex\big)$, i.e. $w_1\in
AC^d\big([0,+\infty),\UH\big)$. By a similar argument, $w_2(t)\in
AC^d\big([0,+\infty),\UH\big)$. The proof is now complete.
\end{proof}

\subsection{Proofs}\label{S_proofs1}\nopagebreak
\begin{proof}[\bf Proof of Theorem~\ref{TH_1punto}.] We have to show that $(\varphi_{s,t})$ satisfies EF3.
To this end we take advantage of Lemma~\ref{LM_lifting_evolut_fam} stating
existence of an $L^d$-evolution family $(\Phi_{s,t})$ in~$\UH$ such
that~\eqref{EQ_G-phi0} holds. Below we use the notation introduced in the
statement and proof of this lemma.

Equality~\eqref{EQ_G-phi0} can be written in the following form:
\begin{equation}\label{EQ_G-phi0shift}
\exp\circ\,\phi_{s,t}=\varphi_{s,t}\circ\exp\quad\text{for any $s\ge0$ and any $t\ge s$},
\end{equation}
where $\phi_{s,t}$ stands for $Q(\cdot,\omega_t)\circ\Phi_{s,t}\circ R(\cdot,\omega_s):\,\mathbb
S_{r(s)}\to\mathbb S_{r(t)}$. Clearly it is sufficient to prove the following
\Unit{\underline{Claim 1}.} \it%
 For any closed interval $I:=[S,T]\subset[0,+\infty)$ and any compact set $K\subset\mathbb
 S_{r(S)}$ there exists a non-negative function $k_{K,I}\in L^d\big(I,\Real\big)$ such that
 \begin{equation}\label{EQ_phi-claim}
 \|\phi_{s,t}-\phi_{s,u}\|_K\le\int_{u}^{t}k_{K,I}(\xi)d\xi
 \end{equation}
 for any $s,u,t\in I$ such that $s\le u\le t$.
\endstep%
To prove the above claim we fix $I:=[S,T]\subset[0,+\infty)$ and a compact set $K\subset\mathbb
 S_{r(S)}$  and consider the set
$K_1:=\cup_{s\in I}R(K,\omega_s)$. Since $R(\zeta,\omega_s)$ is jointly
continuous in $\zeta$ and~$s$ on $\mathbb S_{r(S)}\times I$, the set
$K_1\subset \UH$ is compact. Furthermore, it follows from Lemma~\ref{LM_contD}
that $K_2:=\cup_{S\le s\le t\le T}\Phi_{s,t}(K_1)$ is a compact set in~$\UH$.
Finally, the function $Q(w,\omega)$ is holomorphic in a neighborhood of
$K_2\times[0,+\infty)$ and the function $\tau\mapsto\omega_\tau$ is continuous
on $I$. Hence there exists a constant $C_1=C_1(K_2,I)>0$ such that
$\|Q(\cdot,\omega_t)-Q(\cdot,\omega_u)\|_{K_2}\le C_1|\omega_t-\omega_u|$ for
any $t,u\in I$. By the same reason there exists a constant $C_2=C_2(K_2,I)>0$
such that
$\|Q(\cdot,\omega_u)\circ\Phi_{s,t}-Q(\cdot,\omega_u)\circ\Phi_{s,u}\|_{K_1}\le
C_2\|\Phi_{s,t}-\Phi_{s,u}\|_{K_1}$ for all $u\in I$.

Now we can estimate the left-hand side in~\eqref{EQ_phi-claim} for any $s,u,t\in I$, $s\le u\le t$,
as follows:
\begin{multline*}
 \|\phi_{s,t}-\phi_{s,u}\|_K=%
 \big\|Q(\cdot,\omega_t)\circ\Phi_{s,t}\circ R(\cdot,\omega_s)-
 Q(\cdot,\omega_u)\circ\Phi_{s,u}\circ R(\cdot,\omega_s)\big\|_K\le
\\\vphantom{\int}
 \big\|Q(\cdot,\omega_t)\circ\Phi_{s,t}-
 Q(\cdot,\omega_u)\circ\Phi_{s,u}\big\|_{K_1}\le
 \big\|Q(\cdot,\omega_t)-Q(\cdot,\omega_u)\big\|_{K_2}
\\
 +\big\|Q(\cdot,\omega_u)\circ\Phi_{s,t}-Q(\cdot,\omega_u)\circ\Phi_{s,u}\big\|_{K_1}\le
 C_1|\omega_t-\omega_u|+C_2\|\Phi_{s,t}-\Phi_{s,u}\|_{K_1}.
\end{multline*}
Thus our claim follows from assertion~(ii) of Theorem~\ref{TH_SC-EF-Diff} and from the fact that by
Definition~\ref{def-cansys}, $\tau\mapsto\omega_\tau=\omega(r(t))$ is of class $AC^d$. The proof is
finished.
\end{proof}

\begin{proof}[\bf Proof of Proposition~\ref{PR_uniformity}.]
By condition, $(\varphi_{s,t})$ is an evolution family of order~$d$ over a canonical
$L^d$-system~$(D_t)$. Hence, obviously, it satisfies the conditions of Theorem~\ref{TH_1punto}.
Therefore the statement of Claim~1 in the proof of Theorem~\ref{TH_1punto} is true. Now
Proposition~\ref{PR_uniformity} follows easily from~\eqref{EQ_G-phi0shift}, because the
map~$z\mapsto\exp z$ contracts the Euclidian metric in $\mathbb S_r$ for any $r\in[0,1)$.
\end{proof}

\begin{proof}[\bf Proof of Theorem~\ref{TH_lifting_evolut_fam}]
Let us first construct an $L^d$-evolution family~$(\Psi_{s,t})$
satisfying~\eqref{EQ_G-phi}. Obviously, according to
Lemma~\ref{LM_evol-fam-classM} the family $(\varphi_{s,t})$ fulfills the
conditions of Theorem~\ref{TH_1punto}. Hence by
Lemma~\ref{LM_lifting_evolut_fam}, there exists an $L^d$-evolution family
$(\Phi_{s,t})$ in $\UH$ such that~\eqref{EQ_G-phi0} holds for all $s\ge0$ and
all $t\ge s$. Recall that for each $\tau\ge0$, $B_\tau=\exp\circ
Q(\cdot,\omega_\tau)$, where $Q(\cdot,\omega_\tau)$ maps $\UH$ conformally
onto~$\mathbb S_{r(\tau)}$. Therefore, bearing in mind $r(\tau)>0$ for any
$\tau\ge0$, we have $B_\tau(w)=W_\tau\big(Q(w,\omega_\tau)/\log r(\tau)\big)$
for all $\tau\ge0$ and all $w\in\mathbb H$. It follows, that the family
$(\Psi_{s,t})$, defined by
\begin{eqnarray*}
&&\Psi_{s,t}:=P_t\circ\Phi_{s,t}\circ P_s^{-1},\quad s\ge0,~t\ge s,\\ \text{where} &&
P_\tau(w):=\frac{Q(w,\omega_\tau)}{\log r(\tau)}~~~\text{ for all~$w\in\mathbb H$ and $\tau\ge0$},
\end{eqnarray*}
satisfies equality~\eqref{EQ_G-phi}. It is also easy to see that $(\Psi_{s,t})$ fulfills conditions
EF1 and EF2. Hence, according to Proposition~\ref{PR_2trajectories}, to prove the existence
statement of Theorem~\ref{TH_lifting_evolut_fam} it remains to check that for any
$\zeta_0\in\mathbb S$ the function $[0,+\infty)\ni t\mapsto \Psi_{0,t}(\zeta)$ belongs to
$AC^d\big([0,+\infty),\Complex\big)$.

Fix any $T>0$ and any $\zeta\in\mathbb S$. Denote $w:=P_0^{-1}(\zeta)$. Since
$t\mapsto\Phi_{0,t}(w)$ is continuous,  the set $K:=\big\{\Phi_{0,t}(w):t\in[0,T]\big\}$ is
compact. Hence there exists a positive constant $C=C(\zeta,T)>0$ such that
$\big|P_t(w_2)-P_u(w_1)\big|\le C\big(|w_2-w_1|+|r(t)-r(u)|\big)$ for all $w_1, w_2\in K$ and all
$u,t\in[0,T]$. Therefore,
$$
\big|\Psi_{0,t}(\zeta)-\Psi_{0,u}(\zeta)\big|\le
C\big(|r(t)-r(u)|+|\Phi_{0,t}(w)-\Phi_{0,u}(w)|\big),
$$
whenever $u,t\in[0,T]$. The function $[0,T]\ni t\mapsto\Phi_{0,t}(w)$ is of class~$AC^d$ by EF3,
while $[0,T]\ni t\mapsto r(t)$ is of class~$AC^d$ by definition. Together with the above inequality
this means that $[0,T]\ni t\mapsto \Psi_{0,t}(\zeta)$ is also of class~$AC^d$, which completes the
proof of the fact that~$(\Psi_{s,t})$ is an $L^d$-evolution family in~$\mathbb S$.

It remains to show the uniqueness of $(\Psi_{s,t})$. To this end we fix an
arbitrary $\zeta_0\in\mathbb S$ and take $w_0:=W_0(\zeta_0)\in\mathbb
A_{r(0)}$. Now denote $w(t):=\varphi_{0,t}(w_0)$ and
$\zeta(t):=\Psi_{0,t}(\zeta_0)$ for all ${t\ge0}$. Then
$W_t\big(\zeta(t)\big)=w(t)$ for any $t\ge0$. The mapping $W_t$ is a covering
map of~$\mathbb S$ onto $\mathbb A_{r(t)}$. Hence for any $s\ge0$ and any $t\ge
s$, the function $\Psi_{s,t}$ is the lifting of the mapping~$\varphi_{s,t}\circ
W_s$ w.r.t. $W_t$ that takes $\zeta(s)$ to $\zeta(t)$. It follows that the
uniqueness of the family $(\Psi_{s,t})$ is implied by the uniqueness of the
continuous function $[0,+\infty)\ni t\mapsto \zeta(t)\in\mathbb S$ such that
$\zeta(0)=\zeta_0$ and $W_t\big(\zeta(t)\big)=w(t)$ for all~$t\ge0$. Such a
function $t\mapsto\zeta(t)$ is unique because, according to the definition of
$W_t$, the function $t\mapsto \zeta(t)\log r(t)$ is a lifting of $t\mapsto
w(t)$ w.r.t. $\exp\!:\mathbb S_0\to\UD^*$. With the value $\zeta(0)\log r(0)$
being fixed, this completes the proof.
\end{proof}

\section{Evolution families and differential equations}\label{S_EF-sWHVF}

This section contains our main results. We will establish a one-to-one
correspondence between evolution families over canonical domain systems and
semicomplete weak holomorphic vector fields, analogous to the correspondence
between evolution families and Herglotz vector fields in the unit
disk~\cite{BCM1} and in complex hyperbolic manifolds~\cite{BCM2}. Moreover, we
will give a precise constructive description of semicomplete weak holomorphic
vector fields which appear in our setting for the non-degenerate case. At the
end of the section we consider the degenerate case, which turns out to be
reducible to the case of evolution families in the unit disk.

Throughout this section we fix arbitrary $d\in[1,+\infty]$ and some canonical
domain system~$\big(D_t\big):=\big(\mathbb A_{r(t)}\big)$ of order~$d$. The set
$\mathcal D:=\{(z,t):\,t\ge0,\,z\in D_t\}$ is a relatively open subset of
$\Complex\times E$, where $E:=[0,+\infty)$. So we can apply results of
Section~\ref{S_ODE} to deduce following

\begin{theorem}\label{TH_EF<->sWHVF} The following two assertions hold:
\begin{mylist}
\item[(A)] For any $L^d$-evolution family~$(\varphi_{s,t})$ over the canonical
domain system~$(D_t)$ there exists an essentially unique semicomplete weak
holomorphic vector field $G:\mathcal D\to\Complex$ of order~$d$ and a null-set
$N\subset [0,+\infty)$ such that for all $s\ge0$ the following statements hold:
\end{mylist}
\begin{mylist}
\item[(i)] the mapping $[s,+\infty)\ni t\mapsto \varphi_{s,t}\in\Hol(D_s,\Complex)$
is locally absolutely continuous;
\item[(ii)] the mapping $[s,+\infty)\ni t\mapsto \varphi_{s,t}\in\Hol(D_s,\Complex)$ is
differentiable for all $t\in[s,+\infty)\setminus N$;
\item[(iii)] $d\varphi_{s,t}/dt=G(\cdot,t)\circ\varphi_{s,t}$ for all $t\in[s,+\infty)\setminus N$.
\end{mylist}\vskip3mm

\begin{mylist}
\item[(B)] For any semicomplete weak holomorphic vector field ${G:\mathcal D\to\Complex}$
of order~$d$ the formula $\varphi_{s,t}(z):=w^*_s(z,t)$, $t\ge s\ge0$, $z\in
D_s$, where $w_s^*(z,\cdot)$ is the unique non-extendable solution to the
initial value problem
\begin{equation}\label{EQ_genGolKom}%
 \dot w=G(w,t),\quad w(s)=z,
\end{equation}%
defines an $L^d$-evolution family over the canonical domain system~$(D_t)$.
\end{mylist}
\end{theorem}
In the situation of the above theorem we will say that $G$ is the  {\it vector
field corresponding to} the evolution family~$(\varphi_{s,t})$. The phrase {\it
essentially unique} in this theorem means that for any two vector fields $G_1$
and $G_2$ corresponding to the same evolution family,
$G_1(\cdot,t)=G_2(\cdot,t)$ for a.e. $t\ge0$.
\begin{proof}
Assertion~(B) of the theorem follows readily from
Proposition~\ref{PR_sol_semicomplete}. According to
Theorem~\ref{TH_sol_semicomplete}, in order to prove~(A), we only have to check
that assertions (i)\,--\,(iv) of~Proposition~\ref{PR_sol_semicomplete} hold for
any $L^d$-evolution family~$(\varphi_{s,t})$ over~$(D_t)$. The first three of
them hold by the very definition of an evolution family over a canonical domain
system. To proof assertion (iv) of~Proposition~\ref{PR_sol_semicomplete} we fix
any compact set $K\subset\mathcal D$ and any $T>\max\pr_{\mathbb R}(K)$. Since
$K$ is compact, there exist finite sequences $(S_j)_{j=1}^n\subset[0,T]$ and
$(K_j)_{j=1}^n$ such that $K_j\subset D_{S_j}$ is a compact set for each
$j=1,\ldots,n$ and $K\subset \cup_{j=1}^nK_j\times[S_j,T]$. Apply now
Proposition~\ref{PR_uniformity} with $K_j$ and $I_j:=[S_j,T]$ substituted for
$K$ and $I$, respectively. Further, extend the function $k_{K_j,I_j}$ by zero
to $[0,S_j)$ and define $k_K^T:=\sum_{j=1}^nk_{K_j,I_j}\in
L^d\big([0,T],\Real\big)$. Then, for any $(z,s)\in K$ and any $u,t\in[s,T]$
with~${u\le t}$,
$$%
|\varphi_{s,t}(z)-\varphi_{s,u}(z)|\le \int_u^t k^T_K(\xi)\,d\xi.
$$%
Take any increasing sequence $(T_n)$ such that $T_n>\max\pr_{\mathbb R}(K)$ for
all $n\in\Natural$ and $T_n\to+\infty$ as $n\to+\infty$. Now we finish the
proof of assertion~(iv) of~Proposition~\ref{PR_sol_semicomplete} by setting
$\tilde
k_K:=k_K^{T_1}\chi_{[0,T_1)}+\sum_{n=2}^{+\infty}k_K^{T_n}\chi_{[T_{n-1},T_n)}$.

Finally, the fact that $G$ is essentially unique follows from statement~(iii) of the theorem with
$s:=t$. The proof is finished.
\end{proof}

\subsection{Semicomplete weak holomorphic vector fields in non-degenerate case}\label{SS_charact}
In this section we are going to give a precise constructive description of
semicomplete weak holomorphic vector fields in $\mathcal
D:=\{(z,t):\,t\ge0,\,z\in D_t\}$  for the case when the canonical domain
system~$(D_t)$ is non-degenerate, i.\,e., when $r(t)>0$ for all $t\ge0$.

In order to state the main result of the section we need some notation. First of all, for $r>0$ we
denote by $\mathcal K_r:\mathbb A_r\to\Complex$ the so called {\it Villat kernel}, which can be
defined by the following formula (see, e.\,g.,~\cite{GoluzinM} or~\cite[\S\,V.1]{Aleksandrov}):
\begin{equation}\label{EQ_Villat_kernel}
\mathcal
K_r(z):=\lim\limits_{n\to+\infty}\sum\limits_{\nu=-n}^n\frac{1+r^{2\nu}z}{1-r^{2\nu}z}=\frac{1+z}{1-z}+\sum_{\nu=1}^{+\infty}
\left(\frac{1+r^{2\nu}z}{1-r^{2\nu}z}+\frac{1+z/r^{2\nu}}{1-z/r^{2\nu}}\right).
\end{equation}
The Villat kernel plays the same role for the Function Theory in the annulus as the Schwartz kernel
$\mathcal K_0(z):=(1+z)/(1-z)$ in the unit disk. Namely, for any function $f\in\Hol(\mathbb
A_r,\Complex)$ which is continuous in $\overline{\mathbb A_r}$, see
e.g.\,\cite[Theorem~2.2.10]{Vaitsiakhovich},
\begin{multline}\label{EQ_theVillatFormula}
f(z)=\int_\UC\mathcal K_r(z\xi^{-1})\Re f(\xi)\,\frac{|d\xi|}{2\pi}+\int_\UC \big[\mathcal
K_r(r\xi/z)-1\big]\Re f(r\xi)\,\frac{|d\xi|}{2\pi}\\+i\int_{\UC}\Im
f(\rho\xi)\frac{|d\xi|}{2\pi}\quad \text{for all } z\in\mathbb A_r,~\rho\in[r,1].
\end{multline}
\REM{
\begin{remark}\label{RM_Villat_analytic}
For each fixed $r\in(0,1)$ series~\eqref{EQ_Villat_kernel} converges locally uniformly in
$\Complex^*\setminus\{r^{2\nu}:\nu\in\mathbb Z\}$. It follows that $\mathcal K_r$ is holomorphic in
$\Complex^*$ except the poles at the points ${z_\nu:=r^{2\nu}}$, $\nu\in\mathbb Z$. In a similar
way we see that for a fixed $z\in\Complex$ the function $r\mapsto\mathcal K_r(z)$ is holomorphic
aside the set $\{0\}\cup\{r:r^{2\nu}=z\text{~for~some~$\nu\in\mathbb Z$}\}$, with $r=0$ being its
removable singularity.
\end{remark}
}

\begin{remark}\label{RM_VillatKernel}
Fix any $r\in(0,1)$ and $\rho\in(r,1)$. It is known (see, e.\,g., \cite{GoluzinM} or
\cite[\S\,V.1]{Aleksandrov}) that
\begin{align*}
\max_{\theta\in\Real}\Re \mathcal K_{r}(e^{i\theta}\rho)=\mathcal K_r(\rho)>1,&&
\min_{\theta\in\Real}\Re \mathcal K_{r}(e^{i\theta}\rho)=\mathcal K_r(-\rho)\in(0,1).
\end{align*}
Moreover, $\mathcal K_r(x)$ is increasing on $[-1,-r]$ and on $[r,1)$ with $\mathcal K_r(\pm r)=1$
and $\mathcal K_r(-1)=0$.
\end{remark}

\begin{definition}
Let $r\in(0,1)$. By {\it the class $\ParClass_r$} we will mean the collection of all functions
$p\in\Hol(\mathbb A_r,\Complex)$ having the following integral representation
\begin{equation}\label{EQ_represV}
p(z)=\int_\UC\mathcal K_r(z/\xi)d\mu_1(\xi)+\int_\UC\big[1-\mathcal
K_r(r\xi/z)\big]d\mu_2(\xi),\quad z\in\mathbb A_r,
\end{equation}
where $\mu_1$ and $\mu_2$ are positive Borel measures on the unit circle~$\UC$ subject to the
condition~$\mu_1(\UC)+\mu_2(\UC)=1$.
\end{definition}

\begin{remark}\label{RM_mu1mu2-unique}
>From the proof of~\cite[Theorem~1]{Zmorovich} it is evident that given
$p\in\ParClass_r$, the measures $\mu_1$ and $\mu_2$ in
representation~\eqref{EQ_represV} are unique.
\end{remark}

Let $F\in\Hol(\mathbb A_{r},\Complex)$ for some $r\in(0,1)$. Denote by $\mathcal N(F)$ the free
term in the Laurent expansion of~$F$,
$$
\mathcal N(F):=\int_\UC F(\rho\xi)\frac{|d\xi|}{2\pi},\quad \rho\in(r,1).
$$

\begin{remark}\label{RM_N}
Since $\mathcal N(\mathcal K_r)$=1, we have $\mathcal N(p)=\mu_1(\UC)\ge0$ for any $p\in
\ParClass_{r}$.
\end{remark}

Now we can formulate the main result of this section.
\begin{theorem}\label{TH_semi-char-non-deg}
Let $d\in[1,+\infty]$ and let $\big(D_t\big)=\big(\mathbb A_{r(t)}\big)$ be a
non-degenerate canonical domain system of order~$d$. Then a function
$G:\mathcal D\to\Complex$, where $\mathcal D:=\{(z,t):\,t\ge0,\,z\in D_t\}$, is
a semicomplete weak holomorphic vector field of order~$d$ if and only if there
exist functions $p:\mathcal D\to\Complex$ and $C:[0,+\infty)\to\Real$ such
that:
\begin{mylist}
\item[(i)] $G(w,t)=w\big[iC(t)+
r'(t)p(w,t)/r(t)\big]$ for a.e. $t\ge0$ and all $w\in D_t$;
\item[(ii)] for each $w\in D:=\cup_{t\ge0} D_t$ the function $p(w,\cdot)$ is measurable in
$E_w:=\{t\ge0:\,(w,t)\in\mathcal D\}$;
\item[(iii)]  for each $t\ge0$ the function $p(\cdot\,,t)$ belongs to the class $\ParClass_{r(t)}$;
\item[(iv)] $C\in L^d_{\mathrm{loc}}\big([0,+\infty),\Real\big)$.
\end{mylist}
\end{theorem}

The following proposition forms a main block for the proof of sufficiency in
Theorem~\ref{TH_semi-char-non-deg}.
\begin{proposition}\label{PR_VF}
Let $\big(D_t\big)=\big(\mathbb A_{r(t)}\big)$ be a canonical domain system of
order~$d\in[1,+\infty]$ and~$G$ a semicomplete weak holomorphic vector field
in~$\mathcal D:=\{(z,t):\,t\ge0,\,z\in D_t\}$ of the same order. Then, for
a.e.~$t\in E:=\{t\ge0:\,r(t)>0\}$, the function
$G_t:=G(\cdot,t):D_t\to\Complex$ admits the following representation:
\begin{equation}\label{EQ_GtRepresent}
G_{t}(z) = z\left(\frac{r'(t)}{r(t)}p_t(z)+iC_t\right),\quad C_t\in\Real,~ p_t\in\ParClass_{r(t)}.
\end{equation}
\end{proposition}

\begin{remark}\label{RM_p-tilde}
Denote by $p[r,\mu_1,\mu_2]$ the function defined by representation~\eqref{EQ_represV} and let
$\widehat{\mu}$ stand for the push-forward of a Borel measure~$\mu$ on~$\UC$ w.r.t. the map
$z\mapsto \bar z$. Let $r\in(0,1)$ and $p\in\ParClass_r$. By definition, $p=p[r,\mu_1,\mu_2]$ for
some positive Borel measures $\mu_1$ and $\mu_2$ on~$\UC$ subject to the condition
$\mu_1(\UC)+\mu_2(\UC)=1$. It is easy to see that the function $\tilde p(z):=1-p(r/z)$,
$z\in\mathbb A_r$, also belongs to $\ParClass_r$ and $\tilde
p=p[r,\widehat{\mu}_2,\widehat{\mu}_1]$.
\end{remark}

Denote by $\CarClass$ the Carath\'eodory class of all functions $q\in\Hol(\UD,\Complex)$ normalized
by the condition $q(0)=1$ and satisfying for all $z\in\mathbb D$ the inequality $\Re q(z)\ge0$.
This class has many nice properties, one of which can be formulated as follows.
\begin{remark}\label{RM_convex}
The class~$\CarClass$ is a compact convex subset of~$\Hol(\UD,\Complex)$. Moreover, for any
continuous convex functional  $L:\CarClass\to \Real$,
$$
\max_{q\in\CarClass} L(q)=\max_{\theta\in\Real}L(\mathcal K_0^\theta),\quad
\text{where~}\mathcal K_0^\theta(z):=\mathcal K_0(ze^{-i\theta}).
$$

Indeed, to prove the above statement one only has to apply the Krein~--~Milman theorem (see,
e.\,g., \cite[p.\,181]{Pommerenke}) and take into account that the set of all extremal points
of~$\CarClass$ coincides with $\{\mathcal K_0^\theta:\theta\in\Real\}$ (see, e.\,g.,
\cite{Holland}).
\end{remark}

Our proof of Proposition~\ref{PR_VF} takes advantage of the following lemmas. Recall the notation
$\mathbb S:=\{w:\,0<\Re w<1\}$.

\begin{lemma}\label{LM_posREstrip}
Let $p\in\Hol(\mathbb S,\Complex)$. Suppose that $\Re p(w)\ge0$ for all $w\in\mathbb S$. Then for
any $a,b\in\Real$, $a<b$,
\begin{mylist}
\item[(i)]there exists $C_1(a,b,p)>0$ such that
\begin{equation}\label{EQ_rere}
L_1(p,u):=\int_a^b\Re p(u+iv)\,dv\le C_1(a,b,p),\quad \text{\rm for
all~~}u\in(0,1).
\end{equation}
\item[(ii)] there exists $C_2(a,b,p)>0$ such that
\begin{equation}\label{EQ_imim}
L_2(p,u):=\int_a^b|\Im p(u+iv)|\,dv\le-C_2(a,b,p)\log[u(1-u)],\quad \text{\rm
for all~~}u\in(0,1).
\end{equation}
\end{mylist}
\end{lemma}
\begin{proof}
Fix any $a,b\in\Real$, $a>b$. Note that
$$L_1(p,u)=\Re p(1/2) L_1(p_0,u)\hbox{~~~~and~~~~}L_2(p,u)\le
(b-a)|\Im p(1/2)|+\Re p(1/2) L_2(p_0,u)$$ for any $u\in(0,1)$, where $p_0(w):=\big(p(w)-i\Im
p(1/2)\big)/\Re p(1/2)$ if $\Re p(1/2)>0$ and $p_0\equiv 1$ if $\Re p(1/2)=0$.

Therefore, it is sufficient to prove the lemma only for functions $p\in\Hol(\mathbb S,\Complex)$
satisfying $\Re p(w)\ge0$ for all $w\in\mathbb S$ and normalized by~$p(1/2)=1$. So, further on we
suppose that $p$ fulfills these conditions.

Let $F$ stand for the conformal mapping of the unit disk~$\UD$ onto the
strip~$\mathbb S$ normalized by $F(0)=1/2$, $iF'(0)>0$, namely
\begin{equation}\label{EQ_DtoS}
F(z)=\frac{1}{\pi  i}\,\log\left(i\text{  }\frac{1+z}{1-z}\right).
\end{equation}
Consider the function~$q:=p\circ F$. It belongs to the Carath\'eodory
class~$\CarClass$. For each fixed $u\in(0,1)$ the functionals $q\mapsto
L_1^\UD(q):=L_1(q\circ F^{-1},u)$ and $q\mapsto L_2^\UD(q):= L_2(q\circ
F^{-1},u)$ are convex on the class~$\CarClass$. Therefore, according to
Remark~\ref{RM_convex} we can restrict ourselves to the case $q=\mathcal
K_0^\theta$, $\theta\in\Real$. In this case $p=p_\theta:=\mathcal
K_0^\theta\circ F^{-1}$.

Let us fix $\varepsilon\in(0,1/2)$. Since the functions $u\mapsto L_j(p,u)$,
$j=1,2$, are continuous for~$u\in(0,1)$, it is sufficient to prove~(i) and (ii)
for all $u\in U_\varepsilon:=(0,1/2-\varepsilon)\cup(1/2+\varepsilon,1)$.
Denote $I(u):=F^{-1}\big(\{w=u+iv:v\in[a,b]\}\big)$. The function $F$ extends
holomorphically to the closed unit disk minus~$\{\pm1\}$. For any $u\in
U_\varepsilon$ the preimage of the straight line $\{u+iv:v\in\Real\}$ under $F$
is a circular arc joining points $\pm1$, while the preimages of the segments
$[ia,1+ia]$ and $[ib,1+ib]$ are the hyperbolic geodesics symmetric w.r.t. the
real line. In particular, it follows that there exists $A(a,b,\varepsilon)>0$
such that for any $u\in U_\varepsilon$, $|dF(z)|/|d\arg z|<A(a,b,\varepsilon)$
when $z$ moves along $I(u)$. Hence,
\begin{gather}
\label{EQ_reToDsk}
L_1(p_\theta,u)\le A(a,b,\varepsilon)\int_{I(u)}\Re \mathcal K_0(ze^{-i\theta})\,|d\arg z|,\\
\label{EQ_imToDsk} L_2(p_\theta,u)\le A(a,b,\varepsilon)\int_{I(u)}|\Im
\mathcal K_0(ze^{-i\theta})|\,|d\arg z|
\end{gather}
for all $u\in U_\varepsilon$ and all $\theta\in\Real$.

Using again properties of the function~$F$, we conclude that there are positive
constants $B_1(a,b)$ and $B_2(a,b)$ not depending on $u$ such that
\begin{equation}\label{EQ_ToDsk}
\rho_1(u):=1-B_1(a,b)u(1-u) \le |z|\le 1-B_2(a,b) u(1-u)=:\rho_2(u)
\end{equation}
for all $u\in U_\varepsilon$ and all $z\in I(u)$.

 Recall that $$\mathcal P_0(\rho,\alpha):=\Re \mathcal K_0(\rho
e^{i\alpha})= \frac{1-\rho^2}{1+\rho^2-2\rho\cos\alpha}.$$ Hence, according
to~\eqref{EQ_ToDsk} the integrand in the right-hand side of~\eqref{EQ_reToDsk}
can be estimated as follows
$$
\Re \mathcal K_0(ze^{-i\theta})\le \mathcal P_0\big(\rho_2(u),-\theta+\arg
z\big)\,\frac{1-\rho_1(u)}{1-\rho_2(u)}=\mathcal P_0\big(\rho_2(u),-\theta+\arg
z\big)\,\frac{B_1(a,b)}{B_2(a,b)}
$$
for all $z\in I(u)$. It follows that the integral itself is not greater than
$$
\int_0^{2\pi} \mathcal
P_0\big(\rho_2(u),-\theta+\alpha\big)\frac{B_1(a,b)}{B_2(a,b)}\,d\alpha=2\pi\frac{B_1(a,b)}{B_2(a,b)}.
$$
This proves statement~(i).

Statement (ii) can be proved in a similar way if one notices that for each fixed $\alpha\in\Real$
the function $\rho\mapsto |\Im \mathcal K_0(\rho e^{i\alpha})|$ is non-decreasing on~$(0,1)$ and
that
$$
\int_0^{2\pi} |\Im \mathcal K_0(\rho e^{i\alpha})|\,d\alpha=4\log\frac{1+\rho}{1-\rho}.
$$
\end{proof}

\begin{lemma}\label{LM_represent_p}
Let $r\in(0,1)$ and $F\in\Hol(\mathbb A_r,\Complex)$. Suppose that there exist constants
$\alpha\ge0$, $w_0\in\overline{\mathbb S}\setminus\{\infty\}$ and a function $p\in\Hol(\mathbb
S,\Complex)$ such that for all $w\in\mathbb S$,
$$
\Re p(w)\ge0~~~\hbox{and}~~~F\big(W(w)\big)=\alpha w+
p(w)\sin\frac{\pi}{2}(w_0-w)\sin\frac{\pi}{2}(\bar w_0+w),
$$
where $W(w):=\exp(w\log r)$. Then
\begin{equation}\label{EQ_represent_p} F(z)=iC+\alpha\left[\int_\UC\mathcal
K_{r}(z\xi^{-1})d\mu_1(\xi)+\int_\UC\big[1-\mathcal K_{r}(r(t)\xi/z)\big]d\mu_2(\xi)\right],\quad
z\in\mathbb A_r,
\end{equation}
for some constant~$C\in\Real$ and positive Borel measures $\mu_1$ and $\mu_2$ on the unit
circle~$\UC$ subject to the condition~$\mu_1(\UC)+\mu_2(\UC)=1$.
\end{lemma}
\begin{proof}
Let us prove first that
\begin{equation}\label{EQ_bound}
\int_\UC\big|\Re F(\rho \xi)\big|\,|d\xi|<M<+\infty
\end{equation}
for all $\rho\in(r,1)$ and some constant $M>0$ not depending on~$\rho$.

Since $\Re F(W(w))=\alpha \Re w+\Re\kappa(w)\,\Re p(w)-\Im\kappa(w)\,\Im p(w)$, we can estimate the
integral in~\eqref{EQ_bound} from above by the sum $I_1(u)+I_2(u)+I_3(u)$, where
\begin{gather*}
I_1(u):=\int_0^{T}\alpha\, \Re(u+iv)\,dv=\alpha u T,\qquad
I_2(u):=\int_0^T\big|\Re\kappa(u+iv)\,\Re p(u+iv)\big|\,
dv,\\
I_3(u):=\int_0^T\big|\Im\kappa(u+iv)\,\Im p(u+iv)\big|\, dv,
\end{gather*}
where $u:=\log\rho/\log r\in(0,1)$, $T:=2\pi/|\log r|$ and
$\kappa(w):=\sin\frac{\pi}{2}(w_0-w)\sin\frac{\pi}{2}(\bar w_0+w)$.

Now we note that if $w_0=u_0+iv_0$, then
\begin{gather}\label{EQ_REkappa}
\Re \kappa(u+iv)=\frac12\Big[-\cos\pi u_0+\cos\pi u\ch\pi(v_0-v)\Big],\\\label{EQ_IMkappa}%
\Im \kappa(u+iv)=\frac12\sin\pi u \sh\pi(v_0-v).
\end{gather}
In particular, $\Re \kappa(u+iv)$ is bounded for $u\in(0,1)$ and $v\in[0,T]$, while for~$\Im\kappa$
we have
\begin{equation}\label{EQ_imkappa}
\big|\Im\kappa(u+iv)\big|\le 2u(1-u)\,|\sh\pi(v_0-v)|\le e^{\pi|v-v_0|}u(1-u).
\end{equation}

Therefore, using Lemma~\ref{LM_posREstrip} we conclude that $I_2(u)$ is a bounded function of
$u\in(0,1)$ and that $I_3(u)\to0$ as $u\to1-0$ and as $u\to +0$. This proves~\eqref{EQ_bound}. Then
the following Villat\,--\,Stieltjes representation~\cite[Theorem~1]{Zmorovich} takes
place\footnote{The proof of the result we refer to seems to be published only in Russian; for the
formulation of this result in English, see the Zentralblatt review of~\cite{Zmorovich},
Zbl~0074.05701. Earlier Komatu~\cite{KomatuVSt} obtained the Villat\,--\,Stieltjes representation
in a different form, involving the Weierstra{\ss} zeta-function. Connection between these two forms
can be seen using an expansion of the Weierstra{\ss} zeta-function given, e.g.,
in~\cite[Ch.IV\S20]{Akhiezer}.}:
\begin{equation}\label{EQ_Zmorovich}
F(z)=iC+\int_\UC\mathcal K_r(z\xi^{-1})\,d\nu_1(\xi)+\int_\UC \big[\mathcal
K_r(r\xi/z)-1\big]\,d\nu_2(\xi),\quad z\in\mathbb A_r,
\end{equation}
where $\nu_k$, $k=1,2$, are finite signed Borel measures on~$\UC$ with $\nu_1(\UC)=\nu_2(\UC)$, and
$C\in\Real$.

The proof of~\eqref{EQ_Zmorovich}, given by Zmorovich in~\cite{Zmorovich}, is similar to the proof
of the Herglotz representation theorem, which can be found in~\cite[\S1.9]{Duren}. Namely, for
$\rho\in(r,1)$ and $t\in[0,2\pi]$ we write
$$
\nu^\pm_\rho(t):=\frac1{2\pi}\int_0^{t}F^\pm(\rho,\theta)d\theta,\quad\text{where~}
F^{\pm}(\rho,\theta):=\frac{|\Re F(\rho e^{i\theta})|\pm\Re F(\rho e^{i\theta})}2.
$$
The functions $\nu^+_\rho$ and $\nu^-_\rho$, $\rho\in(0,1)$, are non-decreasing
and, by~\eqref{EQ_bound}, are uniformly bounded. Hence applying the Helly
selection theorem (see, e.g., \cite[p.\,22]{Duren}), we conclude that there is
a sequence $(\rho_{n})\subset(\sqrt r,1)$ converging to~1 and two signed
measures $\nu_j$, $j=1,2$, on~$\UC$ such that
$d\nu^+_{\rho_{j,n}}(\arg\xi)-d\nu^-_{\rho_{j,n}}(\arg\xi)\to d\nu_j(\xi)$ in
the sense of weak-$\ast$ convergence as $n\to+\infty$, where
$\rho_{1,n}:=\rho_n$, $\rho_{2,n}:=r/\rho_n$. It follows, in particular, that
$\nu_1(\UC)=\nu_2(\UC)$, because $\nu^+_\rho(2\pi)-\nu^-_\rho(2\pi)$ does not
depend on $\rho$ according to the Cauchy integral theorem. To see that $\nu_j$
satisfy~\eqref{EQ_Zmorovich}, Zmorovich used the Villat
formula~\eqref{EQ_theVillatFormula} to represent the function $F$ in the
annulus~$A_n:=\{z:\rho_{2,n}<|z|<\rho_{1,n}\}$ via the values of its real part
on~$\partial A_n$. The weak-$\ast$ convergence mentioned above means that for
any continuous function~$\varphi$ on~$\UC$,
\begin{equation}\label{EQ_meas0}
\int_\UC\varphi\,d\nu_j=\lim_{n\to+\infty}\int_\UC\varphi(\xi)\,\Re F(\rho_{j,n}
\xi)\,dm(\xi),\quad j=1,2,\\
%
%\label{EQ_meas1} \int_\UC\varphi\,d\nu_2=\lim_{\rho\to r+0}\int_\UC\varphi(\xi)\,\Re F(\rho
%\xi)\,dm(\xi),
\end{equation}
where $m$ denotes the normalized standard Lebesgue measure on~$\UC$, $dm(\xi)=|d\xi|/(2\pi)$.

Now consider the integral $J(u):=\int_0^T \varphi(e^{iv\log r})\,\Re p(u+iv)\,\Re\kappa(u+iv)\,dv$,
where ${\varphi:\UC\to[0,1]}$ is a continuous function. Taking into account~\eqref{EQ_rere}, for
each $u\in(0,1)$ we have
\begin{multline}\label{EQ_J}
C_1(0,T)M_1(u)\le%
M_1(u)\int_0^T\Re p(u+iv)\,dv\\\le%
J(u)\le\\%
M_2(u)\int_0^T\Re p(u+iv)\,dv\le%
C_1(0,T)M_2(u),
\end{multline}
where $M_1(u):=\min\{0,\Re\kappa(u+iv):v\in[0,T]\}$ and
$M_2(u):=\max\{0,\Re\kappa(u+iv):v\in[0,T]\}$.

Since $I_1(u)$, $I_3(u)$ and $M_1(u)$ tend to~$0$ as $u\to+0$, we deduce from~\eqref{EQ_meas0} with
$j:=1$ and \eqref{EQ_J} that $\int_{\UC}\varphi\,d\nu_1\ge0$ for any
continuous~$\varphi:\UC\to[0,1]$. It follows that $\nu_1\ge0$.

Analogously, using~\eqref{EQ_meas0} with $j:=2$, we prove that $\nu_2\le \alpha m$. Since
$\nu_1(\UC)=\nu_2(\UC)$, we can conclude that $\nu_1(\UC)\in[0,\alpha]$. Set $\mu_1:=\nu_1/\alpha$
and $\mu_2:=m-\nu_2/\alpha$. Clearly, $\mu_1$ and $\mu_2$ are positive Borel measures on~$\UC$ and
$\mu_1(\UC)+\mu_2(\UC)=1$. Now let us notice that the second integral in~\eqref{EQ_Zmorovich} does
not change in value if one adds to the measure $\nu_2$ any constant multiple of~$m$, because
$\int_\UC\mathcal K_r(\rho\xi)\,dm(\xi)=1$ for any $\rho\in(r,1)$. With this remark one immediately
obtains representation~\eqref{EQ_represent_p} from \eqref{EQ_Zmorovich}. The proof is finished.
\end{proof}

\begin{lemma}\label{LM_fromJulia}
Let $p\in\Hol(\mathbb S,\Complex)$ and $\Re p(w)\ge0$ for all $w\in\mathbb S$. Then there exist
finite non-negative limits
$$
A:=\lim_{v\to+\infty}e^{-\pi v}p(1/2+iv),\quad B:=\lim_{v\to-\infty}e^{\pi v}p(1/2+iv).
$$
\end{lemma}
\begin{proof}
For $z$ in the upper half-plane $\mathbb H:=\{z:\Im z>0\}$ write
$$
p_1(z):=ip\left(\frac{\log z}{i\pi}\right),\quad p_2(z):=p_1(-1/z),
$$
where the branch of $\log$ is chosen by setting $\log i=i\pi/2$. Then the limits in the statement
of lemma will take the following form
$$
A=\lim_{y\to+\infty}\frac{p_2(iy)}{iy},\quad B=\lim_{y\to+\infty}\frac{p_1(iy)}{iy}.
$$
Notice that $p_1,p_2\in\Hol(\UH,\overline\UH)$ and apply the classical result about existence of
the angular derivative at~$\infty$, see, e.\,g., \cite[Ch.\,IV Sect.\,26]{Valiron}.
\end{proof}

\begin{proof}[\bf Proof of Proposition~\ref{PR_VF}] For simplicity we assume that
$r(t)>0$ for all $t\ge0$. However, the argument below can be easily adapted for the general case of
a canonical domain system~$(D_t)$ of mixed type.

By Theorem~\ref{TH_EF<->sWHVF}\,-(iv) there exists an $L^d$-evolution family over~$(D_t)$ and a
null set ${N_1\subset[0,+\infty)}$ such that for every $s\ge0$,
$$
\frac{d\varphi_{s,t}(z)}{dt}=G_t(\varphi_{s,t}(z))\qquad \text{for all $z\in D_s$ and
all~$t\in[s,+\infty)\setminus N_1$.}
$$
By Theorem~\ref{TH_lifting_evolut_fam} there exists an $L^d$-evolution family $(\Psi_{s,t})$
in~$\mathbb S$ such that $\varphi_{s,t}\circ W_s=W_t\circ\Psi_{s,t}$ for all $s\ge0$ and all~$t\ge
s$, where $W_\tau(w):=\exp(w\log r(\tau))$ for all $\tau\ge0$ and all~$w\in\mathbb S$. Further by
Theorem~\ref{TH_SC-EF-Diff}, there exists a null-set $N_2\subset[0,+\infty)$ such that for every
$s\ge0$,
$$%
\frac{d\Psi_{s,t}(w)}{dt}=\tilde G_t(\Psi_{s,t}(w))\qquad \text{for all $w\in\mathbb S$ and
all~$t\in[s,+\infty)\setminus N_2$,}
$$%
where $\tilde G_t:\mathbb S\to\Complex$ is an infinitesimal generator for each
$t\in[0,+\infty)\setminus N_2$. Finally, the mapping $t\mapsto r(t)$ is of class~ $AC^d$. Hence
$t\mapsto \log r(t)$ is differentiable for all $t\ge0$ aside some null-set~$N_3$. Thus setting in
the above equalities $z:=W_s(w)$ and then letting $s:=t$, we conclude that for all
$t\in[0,+\infty)\setminus N$, where $N:=N_1\cup N_2\cup N_3$, and all $w\in\mathbb S$,
\begin{equation}\label{EQ_Gt}
G_t\big(W_t(w)\big)=W_t(w)\left[w\frac{r'(t)}{r(t)}+\tilde G_t\big(w\big)\log r(t)\right].
\end{equation}

Now to show that $G_t$ is of form~\eqref{EQ_GtRepresent}, we fix any $t\in[0,+\infty)\setminus N$
and take advantage of Berkson~--~Porta representation for infinitesimal generators in the unit
disk~$\UD$. Namely, according to Remark~\ref{RM_conf-inv-HVF} with $F$ given by~\eqref{EQ_DtoS},
$\tilde G_t(w)=F'\big(F^{-1}(w)\big)H\big(F^{-1}(w)\big)$ for all $w\in\mathbb S$, where
$F^{-1}(w)=i\tg\big(\pi (w-1/2)/2\big)$ and $H(z):=(\tau-z)(1-\bar\tau z)p(z)$ for some point
$\tau\in\overline\UD$ and some function $p\in\Hol(\UD,\Complex)$ satisfying $\Re p(z)\ge0$ for all
$z\in\UD$. Writing $w_0:=F^{-1}(\tau)$, we finally get that either
\begin{align}\label{EQ_tildaGyes}
\tilde G_{t}(w)&=\sin[\pi(\bar w_0+w)/2]\sin[\pi(w_0-w)/2]\tilde p_{t}(w),\quad
w_0\in\overline{\mathbb S}\setminus\{\infty\},\\
\nonumber\text{~\hskip-4em if~$\tau\in\overline\UD\setminus\{\pm1\}$, or}\\
\label{EQ_tildaGno} \tilde G_{t}(w)&=e^{\pm i\pi w}\tilde p_{t}(w)\quad
\text{if~$\tau=\pm1$},
\end{align}
where in both cases $\tilde p_t\in\Hol(\mathbb S,\Complex)$ and $\Re \tilde p_t(w)\ge 0$ for all
$w\in\mathbb S$.

Assume first that equality~\eqref{EQ_tildaGyes} takes place. Applying Lemma~\ref{LM_represent_p}
for $r:=r(t)$, $F(z):=-G_t(z)/z$, $\alpha:=-r'(t)/r(t)$ and $p(w):=-\tilde p_t(w)\log r(t)$ we
obtain formula~\eqref{EQ_GtRepresent}. Thus the proof is finished in this case.

Assume now that $\tilde G_t$ is given by~\eqref{EQ_tildaGno}. Then, on the one
hand, by Lemma~\ref{LM_fromJulia} the function $J(v):=\tilde G_t(iv+1/2)$,
$v\in\Real$, should have a finite purely imaginary limit for $v\to+\infty$ or
for $v\to-\infty$. On the other hand, from~\eqref{EQ_Gt} it follows that $J(v)$
can be written as a periodic function of $v$ plus the linear term
$-ir'(t)v/\big(r(t)\log (t)\big)$. Therefore, $r'(t)=0$ and $J(v)$ is an
imaginary constant. It follows that $G_t(z)=i C_t z$ for all $z\in\mathbb
A_{r(t)}$ and some constant $C_t\in\Real$. Note that $p\equiv0$ belongs to
$\ParClass_r$ for any $r\in(0,1)$. So setting $p_t(w)=0$ for all $w\in D_t$ we
again obtain formula~\eqref{EQ_GtRepresent}. The proof is now complete.
\end{proof}

Now we are going to establish some lemmas which will be used to prove sufficiency in
Theorem~\ref{TH_semi-char-non-deg}. In what follows in this section we assume that
$d\in[1,+\infty]$, $(D_t)$ is a {\it non-degenerate} canonical domains system of order~$d$, and
$\mathcal D:=\{(z,t):\,t\ge0,\,z\in D_t\}$.

\begin{lemma}\label{LM_clV_est}
Let $r\in(0,1)$. Suppose $p\in\ParClass_r$ is given by~\eqref{EQ_represV}. Then for any
$z\in\mathbb A_r$,
\begin{eqnarray}\label{EQ_est_modolus}
|p(z)|&\le&\mu_1(\UC)+\frac{2}{1-r}\left(\frac{|z|}{1-|z|}+\frac{r}{|z|-r}\right),\\
\label{EQ_est_Re}%
-\Re p(z)&\le&\big(\mathcal
K_r\big(r/|z|\big)-1\big)\mu_2(\UC)\le\frac{4r(1-\rho)\mu_2(\UC)}{(\rho-r)(1-r)^2},\quad
\rho:=|z|.
\end{eqnarray}
\end{lemma}
\begin{proof}
Inequality $1-r^{2k}\ge1-r$, where $k\in\Natural$, and the Laurent expansion of $\mathcal K_r$ in
$\mathbb A_r$,
\begin{equation}\label{EQ_Loran}
\mathcal
K_r(z)=\frac{1+z}{1-z}+\sum_{k=1}^{+\infty}\frac{2r^{2k}}{1-r^{2k}}(z^k-z^{-k})=1+2\,\sum_{k=1}^{+\infty}\left(\frac{z^k}{1-r^{2k}}-
\frac{r^{2k}z^{-k}}{1-r^{2k}}\right),
\end{equation}
allows us to estimate $|\mathcal K_r(z)|$ and $|1-\mathcal K_r(r/z)|$, which
together with~\eqref{EQ_represV} leads to~\eqref{EQ_est_modolus}.

The inequality $\Re p(z)\ge \big(1-\mathcal K_r(r/|z|)\big)\mu_2(\UC)$ follows
from~\eqref{EQ_represV} and Remark~\ref{RM_VillatKernel}. Then,  using
again~\eqref{EQ_Loran}, we obtain~\eqref{EQ_est_Re}.
\end{proof}

\begin{lemma}\label{LM_WHVF}
Let $G:\mathcal D\to\Complex$. Suppose that there exist functions $p:\mathcal D\to\Complex$ and
$C:[0,+\infty)\to\Real$ such that conditions (i)\,--\,(iv) are fulfilled. Then $G$ is a weak
holomorphic vector field of order~$d$ in~$\mathcal D$.
\end{lemma}
\begin{proof}
Conditions WHVF1 and WHVF2 from Definition~\ref{D_WHVF} hold trivially.

To prove~WHVF3, fix any compact $K\subset\mathcal D$. Then there exists
$\delta>0$ and $T>0$ such that $r(t)+\delta\le|z|\le1-\delta$ and $t\le T$ for
all $(z,t)\in K$. Applying Lemma~\ref{LM_clV_est}, and taking into account that
$r(T)\le r(t)\le r(0)$ for all $t\in[0,T]$, from \eqref{EQ_est_modolus} we
deduce that
$$|G(z,t)|\le
|C(t)|+\frac{|r'(t)|}{r(T)}\left(1+\frac{4/\delta}{1-r(0)}\right)\quad
\text{for all}~(z,t)\in K.$$ Thus appealing to  condition~(iv) and
Definition~\ref{def-cansys} completes the proof.
\end{proof}

\Unit{{\bf Proof of Theorem~\ref{TH_semi-char-non-deg}: sufficiency.}}\label{PROOF} Suppose that
there exist functions $p:\mathcal D\to\Complex$ and $C:[0,+\infty)\to\Real$ such that conditions
(i)\,--\,(iv) are fulfilled. Then by Lemma~\ref{LM_WHVF}, $G$ is a weak holomorphic vector field of
order~$d$ in~$\mathcal D$. It remains to prove that it is semicomplete. Assume on the contrary that
there exists $(z,s)\in\mathcal D$ such that the domain of definition $J_*(z,s)$ of the unique
non-extendable solution $w_s^*(\cdot,z)$ to the initial value problem~\eqref{EQ_CarODE-IVP} (see
Theorem~\ref{TH_CarODETheory}) is bounded from above. Denote $t^*:=\sup J_*(z,s)$ and
$\rho(t):=\big|w_s^*(t,z)\big|$ for all $t\in [s,t^*)$. Apply Lemma~\ref{LM_clV_est}. According to
(i) and~\eqref{EQ_est_Re},
\begin{equation}\label{EQ_rho_est}
\rho'(t)=\rho(t)\frac{r'(t)}{r(t)}\Re p\big(w_s^*(t,z),t\big)\le-\frac{4 r'(t) \big(1-\rho(t)\big)}
{\big(\rho(t)-r(t)\big)\big(1-r(t)\big)^2}~~~\text{~~for a.e. $t\in[s,t^*)$}.
\end{equation}

We claim that $\rho(t)\le \rho^*(t):=1-\big(1-\rho(s)\big)\exp\alpha\big(r(t)-r(s)\big)$ for all
$t\in[s,t^*)$, where $\alpha:=4\big(\rho(s)-r(s)\big)^{-1}\big(1-r(s)\big)^{-2}$. Indeed, consider
the  function $f(t):=\rho(t)-\rho^*(t)$. This function is locally absolutely continuous
in~$[s,t^*)$. Since $\rho^*(t)\ge \rho(s)$ and $r(t)\le r(s)$ for all $t\in[s,t^*)$,
from~\eqref{EQ_rho_est} it follows that $f'(t)\le0$ for a.e. $t\in[s,t^*)$ such that $f(t)\ge0$.
Bearing in mind that~$f(s)=0$, we therefore conclude that $\rho(t)-\rho^*(t)\le0$ for all
$t\in[s,t^*)$.

It follows that since $t^*<+\infty$, there exists $\delta_1>0$ such that
$\rho(t)=|w_s^*(z,t)|<1-\delta_1$ for all $t\in[s,t^*)$. By Remark~\ref{RM_p-tilde}, $\tilde
p(z,t):=1-p(r(t)/z,t)$ belongs to~$\ParClass_{r(t)}$ for each $t\ge0$. Therefore, choosing
$\rho(t):=|r(t)/w_s^*(z,t)|$ in the above argument, one can conclude also that there exists
$\delta_2>0$ such that $|w_s^*(z,t)|-r(t)>\delta_2$ for all $t\in[s,t^*)$. The fact that these
conclusions contradict Theorem~\ref{TH_CarODETheory}\,-(iii), proves that the vector field $G$ is
semicomplete.
\endstep

\Unit{{\bf Proof of Theorem~\ref{TH_semi-char-non-deg}: necessity.}} Suppose
that $G$ is a semicomplete weak holomorphic vector field of order~$d$. By
Proposition~\ref{PR_VF} there exist functions $p:\mathcal D\to\Complex$ and
$C:[0,+\infty)\to\Real$ satisfying conditions~(i) and~(iii). It remains to
prove~(ii) and (iv).

First of all, we notice that $\Im \mathcal N(p_t)=0$ for all $t\ge0$ because
$\mathcal N(\mathcal K_{r(t)})=1$ (see Remark~\ref{RM_N}). Hence
$C(t)=(1/2\pi)\Im\int_\UC G(\rho\xi,t)/(\rho\xi)\,{|d\xi|}$, where we have
fixed some $\rho\in\big(r(0),1\big)$. Since by definition $G(\cdot,z)$ is
measurable for all $z\in \mathbb A_{r(0)}$ and for each $T>0$ there exists a
non-negative $k_T\in L^d\big([0,T],\Real\big)$ such that $|G(z,t)|\le k_T(t)$
whenever $|z|=\rho$ and $t\in[0,T]$, it follows with the help of the Lebesgue
dominated convergence theorem that $t\mapsto C(t)$ belongs to $L^d_{\rm
loc}\big([0,+\infty),\Real\big)$. This proves~(iv).

We are left with the proof of~(ii). Fix any $s\ge0$ and any $w\in D_s$. On one hand, by the
construction we made in the proof of Proposition~\ref{PR_VF}, $p(w,t)=0$ for a.e. $t\ge s$ such
that $r'(t)=0$. On the other hand, $t\mapsto p(w,t)$ is measurable in~$E_*:=\{t\ge
s:~r'(t)\neq0\}$, because  $t\mapsto r'(t)p(t)/r(t)=G(w,t)/w-iC(t)$ is measurable and $t\mapsto
r(t)$ is locally absolutely continuous in~$[0,+\infty)$. Thus $t\mapsto p(w,t)$ is measurable in
$[s,+\infty)$. Statement (ii) follows now easily. \proofbox\endstep

\subsection{Degenerate case}\label{S_degenerate}

In this section we will show that if a canonical domain system~$(D_t)$ is {\it degenerate},
i.e. $D_t:=\mathbb D^*:=\UD\setminus\{0\}$ for all $t\ge0$, then any evolution
family~$(\varphi_{s,t})$ over~$(D_t)$ can be extended to an evolution family in~$\UD$ with a
common Denjoy~--~Wolff point at the origin. Namely, we prove the following

\begin{proposition}[\protect{compare~\cite[Prop.\,(1.4.30)]{Abate}}]\label{PR_degenerate}
Let $d\in[1,+\infty]$. Suppose $D_t:=\UD^*$ for all $t\ge0$ and  let
$(\varphi_{s,t})$ be an $L^d$-evolution family over~$(D_t)$. Then
$\lim_{z\to0}\varphi_{s,t}(z)=0$ for any~$s\ge0$ and $t\ge s$ and the formula
\begin{equation}\label{EQ_extend}
\phi_{s,t}:=\left\{\begin{array}{ll}\varphi_{s,t}(z),& \hbox{\rm if $z\in\UD^*$},\\0,&
\hbox{\rm if $z=0$,}\end{array}\right.
\end{equation}
defines an $L^d$-evolution family in~$\UD$.
\end{proposition}
\begin{proof}
Fix any $s\ge0$ and any $t\ge s$. Since $\varphi_{s,t}$ is bounded in~$\UD^*$, the origin is its
removable singularity. By Lemma~\ref{LM_evol-fam-classM}, $\varphi_{s,t}\in\mathbb M\big(0,0\big)$.
Hence $\lim_{z\to0}\varphi_{s,t}(z)=0$.

The fact that $(\phi_{s,t})$ satisfies conditions EF1, EF2, and EF3 of
Definition~\ref{D_EV-simply} follows from the corresponding conditions in
Definition~\ref{def-ev} except for EF3 with $z=0$, which holds by the mere fact
that $\phi_{s,t}(0)=0$ for all $s\ge0$ and all $t\ge s$. The proof is complete.
\end{proof}

The converse statement is obvious: {\it if $(\phi_{s,t})$ is an $L^d$-evolution
family in~$\UD$ and ${\phi_{s,t}(0)=0}$ for all $s\ge0$ and all $t\ge s$, then
$(\varphi_{s,t}):=(\phi_{s,t}|_{\UD^*})$ is an $L^d$-evolution family
over~$(D_t)$ with $D_t=\UD^*$ for all $t\ge0$.}

Taking into account Theorem~\ref{TH_EF<->sWHVF} and the above Proposition~\ref{PR_degenerate}, one
can deduce from the results of~\cite{BCM1} a constructive characterization of semicomplete weak
holomorphic vector fields for the degenerate case. Indeed, on the one hand, \cite[Theorem
1.1]{BCM1} establishes the 1-to-1 correspondence between Herglotz vector fields and evolution
families in~$\UD$, while \cite[Theorem~4.8]{BCM1} characterizes Herglotz vector fields in terms of
the Berkson\,--\,Porta representation, see Remarks~\ref{RM_conf-inv-HVF} and~\ref{RM_infDW}. On the
other hand, Theorem~\ref{TH_EF<->sWHVF} establishes the analogous 1-to-1 correspondence between
evolution families and semicomplete weak holomorphic vector fields in the doubly connected
settings. Hence, in view of Proposition~\ref{PR_degenerate}, semicomplete weak holomorphic vector
fields of order~$d\in[1,+\infty]$ over the degenerate canonical domain system are exactly the
Herglotz vector fields, given by \cite[Theorem~4.8]{BCM1} with $\tau(t)=0$ for
a.e.~$t\in[0,+\infty)$.

In this way we obtain the following analogue of~Theorem~\ref{TH_semi-char-non-deg} for the
degenerate case. Let us recall that by $\mathcal C$ we denote the Carath\'eodory class of all
functions $p\in\Hol(\UD,\Complex)$ such that $p(0)=1$ and $\Re p(w)>0$ for all $w\in\UD$.

\begin{proposition}\label{PR_char_semi_degener}
Let $(D_t)$ be a degenerate canonical domain system. Then $G:\mathcal D\to\Complex$, where
$\mathcal D:=\{(z,t):\,t\ge0,\, z\in D_t\}=\UD^*\times[0,+\infty)$, is a semicomplete weak
holomorphic vector field in~$\mathcal D$ of order $d\in[1,+\infty]$ if and only if there exist
functions $\alpha:[0,+\infty)\to[0,+\infty)$, $C:[0,+\infty)\to\Real$, and $p:\mathcal
D\to\Complex$ satisfying the following conditions:
\begin{mylist}
\item[(i)] $G(w,t)=w\big[iC(t)-
\alpha(t) p(w,t)\big]$ for a.e. $t\ge0$ and all $w\in D_t$;
\item[(ii)] for each $w\in D:=\cup_{t\ge0} D_t$ the function $p(w,\cdot)$ is measurable in
$E_w:=\{t\ge0:\,(w,t)\in\mathcal D\}$;
\item[(iii)]  for each $t\ge0$ the function $p(\cdot\,,t)$ belongs to the Carath\'eodory class $\mathcal C$;
\item[(iv)] $C\in L^d_{\mathrm{loc}}\big([0,+\infty),\Real\big)$ and $\alpha\in L^d_{\mathrm{loc}}\big([0,+\infty),
[0,+\infty)\big)$.
\end{mylist}
\end{proposition}

\section{Examples}\label{S_examples}

\begin{example}\label{EX_Rotations}
A set of trivial examples can be obtained by considering static non-degenerate
canonical domain systems $(D_t)$, i.e. canonical domain systems for which $D_t$
does not depend on $t$ and does not coincide with the punctured disk~$\UD^*$,
say $D_t:=\mathbb A_r$ for all $t\ge0$ and some constant $r\in(0,1)$.

In this case by Theorem~\ref{TH_semi-char-non-deg}, the semicomplete weak
holomorphic vector fields of order~$d$ are exactly the functions of the form
$G(w,t)=iC(t)w$, where $C$ belongs to $L^d_{\rm
loc}\big([0,+\infty),\Real\big)$. Hence, according to
Theorem~\ref{TH_EF<->sWHVF}, the $L^d$-evolution families $(\varphi_{s,t})$
over static non-degenerate canonical domain systems are just families of
rotations, $\varphi_{s,t}(z)=ze^{i(\theta(t)-\theta(s))}$, where $\theta\in
AC^d\big([0,+\infty),\Real\big)$.
\end{example}

\begin{example}\label{EX_DW-fail}
According to the classical Denjoy\,--\,Wolff theorem, a self-mapping of the unit disk cannot have
more than one fixed point unless it is the identity map. The infinitesimal version of this
statement implies that a Herglotz vector field $G(z,t)$ in the unit disk (see
Definition~\ref{D_HVF}) cannot have more than one zero for almost every $t\ge0$ such that
$G(\cdot,t)$ does not vanish identically, see Remark~\ref{RM_infDW}.

For mappings of the class $\classM(r_1,r_2)$ the situation is
different. One can have any finite number of fixed points in
$\mathbb A_r$. Now we show an example of an evolution family over
an $L^\infty$-canonical system of annuli sharing an arbitrary
finite number of fixed points.

Let $N\in\Natural$, $r_0\in(0,1)$ and $r_*\in(r_0,1)$. Take
$r(t):=r_0e^{-t}$ and denote $R(t):=r(t)^N$, $R_*:=r_*^N$,
\begin{equation}\label{EQ_ex-alpha} \alpha(t):=\frac{\mathcal
K_{R(t)}\big(R(t)/R_*\big)-1}{\mathcal
K_{R(t)}\big(R_*\big)+\mathcal K_{R(t)}\big(R(t)/R_*\big)-1}.
\end{equation}
According to Remark~\ref{RM_VillatKernel}, $\alpha(t)$ is well-defined for all
$t\ge0$ and satisfies $0<\alpha(t)<1$. Consider the two positive measures on
$\UC$, $\mu_1:=\alpha(t)\mu,$ $\mu_2:=(1-\alpha(t))\mu$,
where$$\mu:=\frac1N\sum_{j=0}^{N-1}\delta_{2\pi j/N},$$ and $\delta_\theta$
denotes the Dirac measure with atom at the point $\xi=e^{i\theta}$.

The corresponding function $p=p_t\in\ParClass_{r(t)}$ given by
representation~\eqref{EQ_represV} is
$$
p_t(z)=\frac1N\left(\alpha(t)\sum_{j=0}^{N-1}\mathcal
K_{r(t)}\big(e^{-2i\pi
j/N}z\big)+\big(1-\alpha(t)\big)\sum_{j=0}^{N-1}\Big[ 1-\mathcal
K_{r(t)}\big(e^{2i\pi j/N}r(t)/z\big)\Big]\right).
$$
Define $F(z,c):=(1+cz)/(1-cz)$, $c>0$. Decomposing $F((cz)^N,1)$ into partial
fractions, we get $\sum_{j=0}^{N-1}F(ze^{2i\pi j/N},c)=N\,F(z^N,c^N)$. It
follows that $\sum_{j=0}^{N-1}\mathcal K_r(ze^{2i\pi j/N})=N\,\mathcal
K_{r^N}(z^N)$ for all $r\in(0,1)$ and $z\in\mathbb A_r$. In particular, we have
\begin{equation}\label{EQ_ex-p}p_t(z)=\alpha(t)\mathcal
K_{R(t)}\big(z^N\big)+\big(1-\alpha(t)\big)\Big[1-\mathcal
K_{R(t)}\big(R(t)/z^N\big)\Big].\end{equation} On the one hand,
by~\eqref{EQ_ex-alpha} and~\eqref{EQ_ex-p}, $p_t(z_j)=0$ for all $t\ge0$ and
$j=0,1,\ldots, N-1$, where $z_j:=r_*e^{2i\pi j/N}$. On the other hand,
$p_t\in\ParClass_{r(t)}$ for each $t\ge0$ and $t\mapsto p_t(z)$ is smooth in
$t$, for each fixed $z\in\mathbb D^*$.

By Theorem~\ref{TH_semi-char-non-deg}, it follows that $G(z,t):=wp_t(w)$ is an
$L^\infty$-semicomplete weak holomorphic vector field over~$(D_t)$ with ${G(z_j,t)=0}$ for all
$j=0,1,\ldots,N-1$ and all $t\ge0$. According to Theorem~\ref{TH_EF<->sWHVF} this vector field
generates an $L^\infty$-evolution family $(\varphi_{s,t})$ over~$(D_t)$ such that
$\varphi_{s,t}(z_j)=z_j$ for all $j=0,1,\ldots,N-1$ and all $s$ and $t$ satisfying $0\le s\le t$.
\end{example}

\begin{example}
Let $\big((D_t),(\varphi_{s,t})\big)$ be a non-degenerate evolution family of
order~$d\in[1,+\infty]$. Define $\tilde \varphi_{s,t}(z):=r(t)/\varphi_{s,t}(r(s)/z)$ for all
$s\ge0$, all $t\ge s$ and all $z\in D_s$. Then it easy to deduce from Theorem~\ref{TH_1punto} that
$(\tilde\varphi_{s,t})$ is also an $L^d$-evolution evolution family over~$(D_t)$. According to
Theorem~\ref{TH_semi-char-non-deg} the semicomplete weak holomorphic vector field~$G$ corresponding
to~$(\varphi_{s,t})$ is given by $G(w,t)=w[iC(t)+r'(t)p(w,t)/r(t)]$ for a.e. $t\ge0$ and all $w\in
D_t$, where the functions $C$ and $p$ are subjects to conditions (ii)\,--\,(iv) from this theorem.
Then the vector field corresponding to the evolution family $(\tilde\varphi_{s,t})$ is given by
$\tilde G(w,t)=w[-iC(t)+r'(r)\tilde p(w,t)/r(t)]$, where by Remark~\ref{RM_p-tilde}, $$\tilde
p(w,t):=1-p\big(r(t)/w\big)=p[r(t),\widehat{\mu_2^t},\widehat{\mu_1^t}](w)\quad \text{for all
$t\ge0$ and all $w\in D_t$}$$ and the families $(\mu_1^t)_{t\ge0}$ and $(\mu_2^t)_{t\ge0}$ are
defined by $p(\cdot,t)=p[r(t),\mu_1^t,\mu_2^t]$ for all $t\ge0$.
\end{example}

\section{Comments on parametric representation of slit mappings}\label{S_Lebedev}

Let $0<m<1<M<+\infty$ and $A:=\{\zeta:m<|\zeta|<M\}$. Denote by~$\mathbb U(A)$ the class of all
univalent holomorphic functions $f:A\to\Complex^*$ such that $f(1)=1$ and for any closed
curve~$\gamma\in A$ the index $I(f\circ\gamma)$ of the origin w.r.t. the curve $f\circ\gamma$
coincides with~$I(\gamma)$.

The following theorem, generalizing results of Komatu~\cite{Komatu}, Goluzin~\cite{GoluzinM} and Li
En Pir~\cite{LiEnPir} on slit mappings of annuli, is due to Lebedev~\cite{Lebedev}\footnote{Note
that the paper~\cite{Lebedev} is a short communication, so it does not contain any proofs.}.

\begin{theorem}[Lebedev]The following statements hold:\vskip.5ex
\begin{mylist}
\item[\rm (A)] Let $f\in \mathbb U(A)$. Suppose that $\partial f(A)$ consists of two disjoint open
Jordan curves one of which extends to~$\infty$ and the other ends at the origin. Then for any
function $\lambda:[0,+\infty)\to[0,1]$, continuous in~$[0,+\infty)$ except for a finite number of
jump discontinuities and subject to the condition
$\int_0^{+\infty}\lambda(t)\,dt=\int_0^{+\infty}\big(1-\lambda(t)\big)dt=+\infty$, there exist
continuous functions $\kappa_j:[0,+\infty)\to\UC$, $j=1,2$, such that
$$%
f(\zeta)=\lim_{t\to+\infty}f(\zeta,t)\quad \text{for all $\zeta\in A$,}
$$%
where $f=f(\zeta,t)$ is the solution to the following initial value problem
\begin{multline}\label{EQ_Lebedev1}
\frac{\dot f}{f}=\lambda(t)\big[\mathcal K_{r(t)}\big(\kappa_1(t)m_t/f\big)-\mathcal
K_{r(t)}\big(\kappa_1(t)m_t\big)\big]\\
-\big(1-\lambda(t)\big)\big[\mathcal K_{r(t)}\big(\kappa_2(t)^{-1}r(t)f/m_t\big)-\mathcal
K_{r(t)}\big(\kappa_2(t)^{-1}r(t)/m_t\big)\big],\quad f|_{t=0}=\zeta,
\end{multline}
$r(t):=e^{-t}m/M$, and $t\mapsto m_t$ is the solution to
\begin{multline}\label{EQ_Lebedev2}
\frac{\dot m_t}{m_t}=-\lambda(t)\Re\mathcal
K_{r(t)}\big(\kappa_1(t)m_t\big)\\-\big(1-\lambda(t)\big)\big[1-\Re\mathcal
K_{r(t)}\big(\kappa_2(t)^{-1}r(t)/m_t\big)\big],\quad m_t|_{t=0}=m.
\end{multline}
\item[\rm (B)] Given any functions $\lambda:[0,+\infty)\to[0,1]$ and $\kappa_j:[0,+\infty)\to\UC$,
$j=1,2$, continuous in~$[0,+\infty)$ except for a finite number of jump discontinuities, the
problem~\eqref{EQ_Lebedev1}\,--\,\eqref{EQ_Lebedev2} has a unique solution $f(\zeta,t)$ defined for
all $t\ge0$ and all $\zeta\in A$, with $f(\cdot,t)\in\mathbb U(A)$ for any~$t\ge0$.\vskip.5ex

\item[\rm (C)] Under the conditions of~(B), $f(\cdot,t)$ converges in $A$ as $t\to+\infty$ to some
function $f\in\mathbb U(A)$.
\end{mylist}
\end{theorem}

In this section we would like to explain how the dynamics of
system~\eqref{EQ_Lebedev1}\,--\,\eqref{EQ_Lebedev2} is connected to the evolution families we
consider in this paper.

First of all, take $r(t):=e^{-t}m/M$ and let $D_t:=\mathbb A_{r(t)}$ for all $t\ge0$. Clearly,
$(D_t)$ is a non-degenerate canonical domain system of order~$d=\infty$. Since
from~\eqref{EQ_Lebedev2} it follows that $(d/dt)\log m_t\le\mathcal K_{r(t)}(r(t)/m_t)-1$ and
$(d/dt)\log(r(t)/m_t)\le \mathcal K_{r(t)}(m_t)-1$ whenever solution ${t\mapsto m_t}$
to~\eqref{EQ_Lebedev2} exists, we can conclude, arguing in the same way as in the proof of the
sufficiency statement of Theorem~\ref{TH_semi-char-non-deg} on page~\pageref{PROOF}, that for any
{\it measurable} functions ${\lambda:[0,+\infty)\to[0,1]}$ and ${\kappa_j:[0,+\infty)\to\UC}$,
$j=1,2$, the initial value problem has a unique solution ${[0,+\infty)\ni t\mapsto m_t}$ and that
$r(t)< m_t<1$ for all $t\ge0$. Now with the change of variables $z:=\zeta/M$, $w:=r(t)f/m_t$, the
initial value problem~\eqref{EQ_Lebedev1} is equivalent to
\begin{align*}
&\dot w=G(w,t),\quad w|_{t=0}=z,\qquad
\text{where}\quad G(w,t):=w[iC(t)-p(w,t)],\\
C(t)&:=\big(1-\lambda(t)\big)\Im \mathcal
K_{r(t)}\big(\kappa_2(t)^{-1}r(t)/m_t\big)-\lambda(t)\Im\mathcal
K_{r(t)}\big(m_t\kappa_1(t)\big),\\
p(\cdot,t)&:=p\big[r(t),\,\big(1-\lambda(t)\big)\nu_2^t,~\lambda(t)\nu_1^t\big]\quad \text{ for all
$t\ge0$ and all $w\in D_t$},
\end{align*}
and $\nu_j^t$, $j=1,2$, $t\ge0$, stands for the Dirac measure on~$\UC$ with the atom at
$\kappa_j(t)$.

It is easy to see that functions $C$ and $p$ satisfy conditions (ii)\,--\,(iv) from
Theorem~\ref{TH_semi-char-non-deg} with $d=+\infty$. Hence $G$ is a semicomplete weak holomorphic
vector fields of order~$+\infty$ and $f(\zeta,t)=m_t\varphi_{0,t}(\zeta/M)/r(t)$ for all $t\ge0$,
where $(\varphi_{s,t})$ is the $L^{\infty}$-evolution family over~$\big(\mathbb A_{r(t)}\big)$
generated by~$G$ in the sense of Theorem~\ref{TH_EF<->sWHVF}\,-(iv). Therefore we conclude that
dynamics of system~\eqref{EQ_Lebedev1}\,--\,\eqref{EQ_Lebedev2} can be regarded as a very special
case of dynamics of evolution families we consider in this paper and that statement~(B) of
Lebedev's theorem follows from our results. In particular, equation~\eqref{EQ_genGolKom} with the
vector field $G$ admitting the representation given in Theorem~\ref{TH_semi-char-non-deg} is a
generalization of the Komatu equation~\cite{Komatu, GoluzinM} (known also as the
Goluzin\,--\,Komatu equation),  since \eqref{EQ_Lebedev1}\,--\,\eqref{EQ_Lebedev2} reduces to the
latter equation when $\lambda\equiv0$ and $m=1$.

\end{document}